\documentclass[oneside, a4paper,11pt,reqno]{amsart}
\usepackage{amssymb,amsmath,amsthm}
\usepackage[T1]{fontenc}
\usepackage{color}

\newcommand{\fdem}{\hspace*{\fill}~$\Box$\par\endtrivlist\unskip}

\newcommand{\Z}{{Z\!\!\!Z}}

\newtheorem{hypo}{Hypothesis}

\newtheorem{prop}[hypo]{Proposition}

\newtheorem{thm}[hypo]{Theorem}

\newtheorem{lem}[hypo]{Lemma}

\def\PP{\mathbb{P}}
\def\RR{\mathbb{R}}
\def\ZZ{\mathbb{Z}}
\def\CC{\mathbb{C}}
\def\EE{\mathbb{E}}

\newcommand {\cro}[1] {\left[ {#1} \right]}

\newcommand {\va}[1] {\left| {#1} \right|}

\newcommand {\floor}[1] {\left\lfloor {#1} \right\rfloor}

\title{Renewal theorems for random walks in random scenery}

\author{Nadine Guillotin-Plantard} 
\address{Institut Camille Jordan, CNRS UMR 5208, Universit\'e de Lyon, Universit\'e Lyon 1, 43, Boulevard du 11 novembre 1918, 69622 Villeurbanne, France.}
\email{nadine.guillotin@univ-lyon1.fr}

\author{Fran\c{c}oise P\`ene}
\address{Universit\'e Europ\'eenne de Bretagne,
D\'epartement de Math\'ematiques, UMR CNRS 6205, 29238 Brest cedex, France}
\email{francoise.pene@univ-brest.fr}

\keywords{Random walk in random scenery; renewal theory; local time; stable distribution\\
This work was partially supported by the french ANR project MEMEMO2}

\begin{document}

\begin{abstract} 
Random walks in random scenery are processes defined by  
$Z_n:=\sum_{k=1}^n\xi_{X_1+...+X_k}$, where $(X_k,k\ge 1)$ and 
$(\xi_y,y\in\mathbb Z)$
are two independent sequences of i.i.d.~random variables. 
We suppose that the distributions of $X_1$ and $\xi_0$ belong
to the normal domain of attraction of strictly stable distributions
with index $\alpha\in[1,2]$ and $\beta\in(0,2)$ respectively. 
We are interested in the asymptotic behaviour as $|a|$ goes
to infinity of quantities of the form
$\sum_{n\ge 1}{\mathbb E}[h(Z_n-a)]$ (when $(Z_n)_n$ is transient) or 
$\sum_{n\ge 1}{\mathbb E}[h(Z_n)-h(Z_n-a)]$ (when $(Z_n)_n$ is recurrent) where $h$ is some complex-valued function defined on $\mathbb{R}$ or $\mathbb{Z}$.
\end{abstract}
\maketitle

\section{Introduction} 
Renewal theorems in probability theory deal with the asymptotic behaviour when $|a|\rightarrow +\infty$ of the potential kernel formally defined as
$$ K_a(h):= \sum_{n=1}^{\infty} \mathbb{E}[h(Z_n-a)]$$
where $h$ is some complex-valued function defined on $\mathbb{R}$ and $(Z_n)_{n\geq 1}$ a real
transient random process. 
The above kernel $K_a(.)$ 
is not well-defined for recurrent process $(Z_n)_{n\geq 1}$, in that case, we would rather 
study the kernel 
$$G_{n,a}(h) := \sum_{k=1}^n \Big\{\mathbb{E}[ h(Z_k)] - \mathbb{E}[ h(Z_k-a)] \Big\}$$ 
for $n$ and $|a|$ large.
In the classical case when $Z_n$ is the sum of $n$ non-centered  
independent and identically distributed real random variables, renewal theorems were proved 
by Erd\"{o}s, Feller and Pollard \cite{EFP}, Blackwell \cite{Bla1,Bla2}, Breiman \cite{Bre}. 
Extensions to multi-dimensional real random walks or additive functionals of Markov chains 
were also obtained (see \cite{Gui} for statements and references). 

In the particular case where the process $(Z_n)_{n\geq 1}$ takes its values in $\mathbb{Z}$ and $h$ is the Dirac function at 0, the study of the corresponding kernels
$$K_a(\delta_0)= \sum_{n=1}^{\infty} \mathbb{P}[Z_n=a]$$
and
$$G_{n,a}(\delta_0)= \sum_{k=1}^n \Big\{\mathbb{P}[Z_k=0] - \mathbb{P}[Z_k=a]\Big\}$$  
have a long history (see \cite{spitzer}). 
In the case of aperiodic recurrent random walks on $\mathbb{Z}$  with finite variance, the potential kernel is known to behave asymptotically as $|a|$ 
when $|a|$ goes to infinity
and, for some particular  random walks as the simple random walk,  an explicit formula can be given  (see Chapter VII in \cite{spitzer}). \\*
In this paper we are interested in renewal theorems for random walk in random scenery (RWRS). Random walk in random scenery  is a simple model of process in disordered
media with long-range correlations. They have been used in a wide
variety of models in physics to study anomalous dispersion in layered
random flows \cite{matheron_demarsily},  diffusion with random sources, 
or spin depolarization in random fields (we refer the reader
to Le Doussal's review paper \cite{ledoussal} for a discussion of these
models). On the mathematical side, motivated by the construction of 
new self-similar processes with stationary increments, 
Kesten and Spitzer \cite{KestenSpitzer} and Borodin  \cite{Borodin, Borodin1} 
introduced RWRS in dimension one and proved functional limit theorems. 
This study has been completed in many works, in particular in
\cite{bolthausen} and \cite{DU}.
These processes are defined as follows. 
We consider two independent sequences $(X_k,k\ge 1)$ and 
$(\xi_y,y\in\mathbb Z)$
of independent identically distributed random variables with values 
in $\mathbb Z$ and $\mathbb R$ respectively.
We define 
$$\forall n\ge 1,\ \ S_n:=\sum_{k=1}^nX_k\ \ \mbox{and}\ \ S_0:=0.$$
The {\it random walk in random scenery} $Z$ is 
then defined for all $n\ge 1$ by
$$Z_n:=\sum_{k=1}^{n}\xi_{S_k}.$$ 
The symbol $\#$ stands for the cardinality of a finite set. Denoting by  
$N_n(y)$ the local time of the random walk $S$~:
$$N_n(y)=\#\big\{k=1,...,n\ :\ S_k=y\big\} \, 
$$
the random variable $Z_n$ can be rewritten as 
\begin{equation} 
\label{re-ecriture-Zn}
Z_n=\sum_{y\in\Z}\xi_yN_n(y).
\end{equation}
The distribution of $\xi_0$ is assumed to belong to the normal 
domain of attraction of a strictly stable distribution 
$\mathcal{S}_{\beta}$ of 
index $\beta\in (0,2]$, with characteristic function $\phi$ given by
\begin{equation}\label{FC}
\phi(u)=e^{-|u|^\beta(A_1+i A_2\  \text{sgn}(u))},\quad u\in\mathbb{R},
\end{equation}
where $0<A_1<\infty$ and $ |A_1^{-1}A_2|\le |\tan (\pi\beta/2)|$. When $\beta=1$, $A_2$ is null.
We will denote by $\varphi_\xi$ the characteristic function of the random variables~$\xi_x$. 
When $\beta > 1$, this implies that $\EE[\xi_0] = 0$. 
Under these conditions, we have, for $\beta \in (0,2]$, 
\begin{equation}
\label{queue}
\forall t > 0 \, , \,\, \PP \cro{ \va{\xi_0} \ge t} \le \frac{C(\beta)}{t^{\beta}} \, .
\end{equation}
Concerning the random walk $(S_n)_{n\geq1}$,  the distribution of $X_1$ 
is assumed to belong to the normal
domain of attraction of a strictly stable distribution  $\mathcal{S}_{\alpha}'$ of 
index $\alpha$. Since, when $\alpha<1$, the behaviour of $(Z_n)_n$ is
very similar to the behaviour of the sum of the $\xi_k$'s, $k=1,\ldots,n$, we restrict ourselves to the study
of the case when $\alpha\in [1,2]$. 
Under the previous assumptions, the following weak convergences hold in the space  of 
c\`adl\`ag real-valued functions 
defined on $[0,\infty)$ and on $\mathbb R$ respectively,  endowed with the 
Skorohod topology ~:
$$\left(n^{-\frac{1}{\alpha}} S_{\lfloor nt\rfloor}\right)_{t\geq 0}   
\mathop{\Longrightarrow}_{n\rightarrow\infty}
^{\mathcal{L}} \left(U(t)\right)_{t\geq 0}$$
$$\mbox{\rm and} \  \  \   \left(n^{-\frac{1}{\beta}} 
\sum_{k=0}^{\lfloor nx\rfloor}\xi_{k }\right)_{x\ge 0}
   \mathop{\Longrightarrow}_{n\rightarrow\infty}^{\mathcal{L}} 
\left(Y(x)\right)_{x\ge 0},$$
where $U$ and $Y$ are two independent L\'evy processes such 
that $U(0)=0$, $Y(0)=0$, 
$U(1)$ has distribution $\mathcal{S}'_{\alpha}$ and $Y(1)$ 
has distribution  $\mathcal{S}_\beta$. 
For $\alpha\in\ ]1,2]$, we will denote by $(L_t(x))_{x\in\mathbb{R},t\geq 0}$ a continuous version with compact support of the local time of the process
 $(U(t))_{t\geq 0}$ and  by $|L|_\beta$ the random variable $\left( \int_{\mathbb R} L_1^
 \beta (x)\, d x \right)^{1/\beta} $. Next let us define 
\begin{equation} 
\label{def-delta}
\delta:=1-\frac 1\alpha+\frac 1{\alpha\beta} = 1 + \frac 1\alpha(\frac{1}{\beta}-1).
\end{equation}
In \cite{KestenSpitzer}, Kesten and Spitzer proved 
the convergence in distribution of $((n^{-\delta} Z_{nt})_{t\ge 0})_n$, when $\alpha>1$, 
to a process $(\Delta_t)_{t\geq 0}$ defined as
$$\Delta_t = \int_{\mathbb{R}} L_t(x) \, dY(x),$$
by considering a process $(Y(-x))_{x\ge 0}$ with the same distribution
as $(Y(x))_{x\ge 0}$ and independent of $U$ and $(Y(x))_{x\ge 0}$.

\noindent In \cite{DU}, Deligiannidis and Utev considered the case when $\alpha=1$ and $\beta=2$
and proved the convergence in distribution of $( (Z_{nt} /\sqrt{n\log(n)})_{t\ge 0})_n$ to
a Brownian motion. This result is got by an adaptation of the proof of the same result 
by Bothausen in \cite{bolthausen} in the case 
when $\beta=2$ and for a square integrable two-dimensional random walk $(S_n)_n$.

\noindent In \cite{FFN}, Castell, Guillotin-Plantard and P\`ene completed the study of the case
$\alpha=1$ by proving the convergence of $( n^{-\frac 1\beta}(\log(n))^{\frac 1\beta-1}
Z_{nt})_{t\ge 0} )_n$ to $c^{\frac 1\beta}(Y(t))_{t\in\mathbb R}$, with 
\begin{equation}\label{cc}
c:=(\pi a_0)^{1-\beta}\Gamma(\beta+1),
\end{equation}
 where $a_0$ is such that
$t\mapsto e^{-a_0|t|}$ is the characteristic function of the limit of
$(n^{-1}S_n)_n$.

\noindent Let us indicate that, when $\alpha\ge 1$, the process
$(Z_n)_n$ is transient (resp. recurrent) if $\beta<1$
(resp. $\beta> 1$) (see \cite{BFFN,KS}).

\noindent We recall the definition of the Fourier transform $\hat h$ as follows.
For every $h:\mathbb R\rightarrow \mathbb C$ (resp. $h:\mathbb Z\rightarrow \mathbb C$)
integrable with respect to the Lebesgue measure on $\mathbb R$ (resp. with
respect to the counting measure on $\mathbb Z$), we denote by $I[h]$
the integral of $h$ and by $\hat h:\mathcal I\rightarrow\mathbb C$ its Fourier transform
defined by
$$\forall x\in{\mathcal I},\ \ \hat h(x):=I[h(\cdot)e^{ix\cdot}],\ \ \mbox{with}\
   {\mathcal I}={\mathbb R}\ \ \mbox{(resp. } {\mathcal I}=[-\pi;\pi]\mbox{)}.$$ 
\subsection{Recurrent case~: $\beta\in [1,2]$} 

We consider two distinct cases:
\begin{itemize}
\item{\bf Lattice case:} The random variables $(\xi_x)_{x\in \ZZ}$ are assumed to be $\mathbb Z$-valued and non-arithmetic i.e. 
$\{u; |\varphi_\xi(u)|=1\}=2\pi\ZZ$.\\*
\noindent The distribution of $\xi_0$ belongs to the normal 
domain of attraction of $\mathcal S_\beta$ with characteristic function $\phi$ given by
(\ref{FC}). 
\item{\bf Strongly non-lattice case:} The random variables $(\xi_x)_{x \in\ZZ}$ are assumed to be 
strongly non-lattice i.e.
$$\limsup_{|u|\rightarrow +\infty}|\varphi_\xi(u)|<1.$$
The distribution of $\xi_0$ belongs to the normal 
domain of attraction of $\mathcal S_\beta$ with characteristic function $\phi$ given by
(\ref{FC}). \end{itemize}
For any $a\in \mathbb{R}$ (resp. $a\in\mathbb Z$), we consider  the  kernel  
$K_{n,a} $ defined as follows~: for any $h : \mathbb{R}\to\mathbb{C}$ 
(resp. $h : \mathbb{Z}\to\mathbb{C}$) in the strongly non-lattice (resp. in the lattice) case, we write 
$$K_{n,a}(h) :=   \sum_{k=1}^n \Big\{ \mathbb{E}[h(Z_k)] -  {\mathbb E}[h(Z_k-a)]\Big\}$$ 
when it is well-defined.

\begin{thm}\label{threc2}
The following assertions hold for every integrable function $h$ on $\mathbb R$ 
with Fourier transform integrable on $\mathbb R$ in the strongly non-lattice case and for every integrable function $h$ on $\mathbb Z$ 
 in the lattice case.
\begin{itemize}
\item when $\alpha>1$ and $\beta>1$,
 $$\lim_{a\rightarrow +\infty}  a^{1-\frac 1 \delta}
    \lim_{n\rightarrow +\infty }   K_{n,a}(h)  =C_1  I[h],$$
with
$$ C_1:=\frac{ \Gamma(\frac{1}{\delta\beta})\Gamma(2-\frac{1}{\delta})\mathbb{E} 
\big[|L|_{\beta}^{-1/\delta}\big]}{\pi\beta (1-\delta)(A_1^2 +A_2^2)^{1/2\delta\beta}}
\sin\left(\frac{1}{\delta} \left(\frac{\pi}{2} -\frac{1}{\beta} \arctan 
\Big(\frac{A_2}{A_1}\Big)\right)\right).$$
\item when $\alpha\ge 1$ and $\beta=1$,
 $$\lim_{a\rightarrow +\infty} ( \log a)^{-1}
    \lim_{n\rightarrow +\infty }   K_{n,a}(h)  =C_2  I[h],$$
with $C_2:=(\pi A_1)^{-1}$.
\item when $\alpha=1$ and $\beta\in(1,2)$,
 $$ \lim_{a \rightarrow +\infty} \left(a^{-1}\log (a^\beta)\right)^{\beta-1}
    \lim_{n\rightarrow +\infty }   K_{n,a}(h)=D_1I[h],$$
with
$$D_1:= \frac{\Gamma(2-\beta) }
   {\pi c (\beta-1) (A_1^2+A_2^2)^{1/2}}
    \sin\left( \frac{\pi\beta}{2} -  \arctan \Big(\frac{A_2}{A_1}\Big) \right).$$
\item when $\alpha=1$ and $\beta=2$, assume that $h$ is even and that the distribution of the $\xi_x's$ is symmetric, then
 $$ \lim_{a\rightarrow +\infty} \left(a^{-1}  \log (a^2)\right) 
    \lim_{n\rightarrow +\infty }   K_{n,a}(h)=D_2I[h],$$
with $D_2:=(2 A_1 c)^{-1} .$
\end{itemize}
\end{thm}
\noindent{\bf Remarks:} 
1- It is worth noticing that since $|A_2/A_1|\leq |\tan( \pi\beta /2)|$,  the constants $C_1$ 
and $D_1$ are strictly positive.\\
2- The limit as $a$ goes to $-\infty$ is not considered in Theorem \ref{threc2} 
and Theorem \ref{threc0} since it can be easily obtained from the limit as $a$ goes to infinity. 
Indeed, the problem is then equivalent to study the limit as $a$ goes to infinity with the 
random variables  $(\xi_x)_x$ replaced by $(-\xi_x)_x$ and the function $h$ by 
$x\mapsto h(-x)$. 
The limits can easily be deduced from the above limit constants by changing  $A_2$ to $-A_2$.
\subsection{Transient case~: $\beta\in (0,1)$}   
\noindent Let ${\mathcal H}_1$ denote the set of all the complex-valued Lebesgue-integrable functions $h$ such that its Fourier transform 
$\hat h$ is continuously differentiable on $\mathbb{R}$, with in addition $\hat h$ and $(\hat h)'$ Lebesgue-integrable. 
\begin{thm} \label{threc0}
Assume that $\alpha\in (1,2]$ and that the characteristic function of the random variable $\xi_0$ is equal to $\phi$ given
by (\ref{FC}). \\*
\noindent Then, for all $h\in{\mathcal H}_1$, we have 
$$\lim_{a\rightarrow +\infty}  a^{1-\frac 1 \delta}
    \sum_{n\ge 1}{\mathbb E}[h(Z_n-a)]= C_0\  I[h]$$
with 
$$C_0:= \frac{\Gamma(\frac{1}{\delta\beta})\Gamma(2-\frac{1}{\delta})\mathbb{E} 
\big[|L|_{\beta}^{-1/\delta}\big]}{\pi\beta (\delta-1) (A_1^2 +A_2^2)^{1/2\delta\beta}}
\sin\left(\frac{1}{\delta} \left(\frac{\pi}{2} -\frac{1}{\beta} \arctan 
\Big(\frac{A_2}{A_1}\Big)\right)\right).$$

\end{thm}

\subsection{Preliminaries to the proofs}
In our proofs, we will use Fourier transforms for some $h:{\mathbb R}\rightarrow{\mathbb C}$
or $h:{\mathbb Z}\rightarrow{\mathbb C}$ and, more precisely,
the following fact
$$2\pi{\mathbb E}[h(Z_n-a)]=\int_{\mathcal I}\hat h(t) {\mathbb E}[e^{i tZ_n}]e^{-iat}\, dt.$$
This will lead us to the study of
$\sum_{n\ge 1}{\mathbb E}[e^{itZ_n}]$. Therefore it will be crucial to observe that we have
$$\forall t\in{\mathbb R},\ \forall n\ge 1,\ {\mathbb E}[e^{it Z_n}]={\mathbb E}
   \left[\prod_{y\in\mathbb Z}e^{it\xi_yN_n(y)}\right]
    ={\mathbb E}\left[\prod_{y\in\mathbb Z}\varphi_\xi(tN_n(y))\right], $$ 
since, taken $(S_k)_{k\le n}$, $(\xi_y)_y$ is a sequence of iid random variables
with characteristic function $\varphi_\xi$. 
Let us notice that, in the particular case when $\xi_0$ has the stable distribution
given by characteristic function (\ref{FC}), the
quantity $\sum_{n\ge 1}{\mathbb E}[e^{itZ_n}]$ is equal to
$$  \psi(t):=\sum_{n\ge 1}{\mathbb E}\bigg[\prod_{y\in\mathbb Z}
   e^{-|t|^\beta N_n(y)^\beta(A_1+i A_2\text{sgn}(t) )}\bigg].$$
Section 2 is devoted to the study of this series thanks to which, 
we prove Theorem \ref{threc2} in Section 3  and Theorem \ref{threc0} in Section 4.
\section{Study of the series $\Psi$}
Let us notice that we have, for every real number $t\neq 0$,
\begin{equation}\label{def-psi}
\psi(t) =\sum_{n\ge 1} {\mathbb E}[e^{-|t|^\beta V_n (A_1+i A_2\text{sgn}(t) )}],
\end{equation}
with $V_n:=\sum_{y\in\mathbb Z} N_n(y)^\beta$. 
Let us observe that $\frac{N_n(y)}{n} \leq (\frac{N_n(y)}{n})^\beta\le\frac {N_n(y)}{n^\beta}$ 
if $\beta \leq 1$, and 
$\frac{ N_n(y)}{n^\beta}\le 
   (\frac{N_n(y)}{n})^\beta \leq \frac{N_n(y)}{n}$ if $\beta > 1$.
Combining this with the fact that $\sum_{y\in\mathbb Z} N_n(y)=n$, we obtain: 
\begin{subequations}
\begin{eqnarray}
& & \beta \leq 1\ \Rightarrow\ n^\beta\leq V_n \leq n  \label{Vn-beta-less1} \\
& &  \beta \geq 1\ \Rightarrow\ n\leq V_n \leq n^\beta  \label{Vn-beta-less2}.
\end{eqnarray}
\end{subequations}
\begin{prop}\label{series-1-trans}
When $\beta\in (0,2]$,  for every $r\in(0,+\infty)$, the function $\psi$ is bounded on the set $\{t\in {\mathbb R} : |t|\geq r\}$.\\*
When $\beta\in (0,1)$,  the function $\psi$ is differentiable on $\mathbb R\setminus\{0\}$, and for every $r\in(0,+\infty)$,  its derivative $\psi'$ is bounded on the set $\{t\in {\mathbb R} : |t|\geq r\}$. 
\end{prop}
\begin{proof}
Let $r>0$. Then: $|t|\geq r\ \Rightarrow\ |\psi(t)| \leq  \sum_{n\geq 1} e^{-A_1 r^\beta n^{1\wedge \beta}}$, so the first assertion is proved. Next, when $\beta\in (0,1)$, since $\sum_{n\geq 1} n\, e^{-A_1(rn)^\beta} < \infty$, it easily follows from Lebesgue's theorem that $\psi$ is differentiable on $\{t\in {\mathbb R} : |t|\geq r\}$, with $|\psi'(t)| \leq  \beta r^{\beta-1}(A_1+|A_2|) \sum_{n\geq 1} n\, e^{-A_1 (rn)^\beta}$ when $|t|\geq r$.  
\end{proof}
In the particular case when $\beta=1$, we have $A_2=0$ and
$$\psi(t)=\frac 1{e^{A_1|t|} -1}\sim_{t\rightarrow 0}\gamma(t),\ \
    \mbox{with}\ \ \gamma(t):={A_1}^{-1}|t|^{-1}.$$
When $\beta\ne 1$, the expression of $\psi(t)$ is not so simple. We will need some estimates
to prove our results. Recall that the constant $c$ is defined in (\ref{cc}) if $\alpha=1$ and set
$$C:= \frac{1}{\delta\beta}Ê\Gamma( \frac{1}{\delta\beta}) {\mathbb E}[ |L|_\beta^{-1/\delta}].  $$
\begin{prop}\label{series-2}
When $\alpha>1, \beta\ne 1$, we have
\begin{equation}
\label{lim-psi} 
\lim_{t\rightarrow 0} \frac{\psi(t)}{\gamma(t)}=1
\end{equation} and
\begin{equation}
\label{lim-psi-deri}
\lim_{t\rightarrow 0} \frac{\psi'(t)}{\gamma'(t)}=1,
\end{equation}
where $\gamma$ is the function defined by
$$ \gamma(t):=C |t|^{-1/\delta}  (A_1+ iA_2\, \mbox{\rm sgn}(t))^{-1/(\delta\beta)}.$$
When $\alpha=1$ and $\beta>1$, we have
\begin{equation}
\label{lim-psi2} 
\lim_{t\rightarrow 0} \frac{\psi(t)}{\gamma(t)}=1,
\end{equation}
where $\gamma$ is the function defined by $$\gamma(t):=\frac{(-\log (|t|^\beta) )
   ^{1-\beta}}{c |t|^{\beta} (A_1+ iA_2\, \mbox{\rm sgn}(t)) }. $$
\end{prop}

\noindent To prove Proposition~\ref{series-2}, we need some preliminaries lemmas. Let us define
\begin{equation}\label{bn}
b_n:=n^{\delta}\ \mbox{if}\ \alpha>1\ \ \ \mbox{and}\ \ \ 
 b_n:= n^{\frac 1\beta}(\log(n))^{1-\frac 1\beta}\ \mbox{if}\ \alpha=1.
\end{equation}
We first recall some facts on the behaviour of the sequence 
$\left(b_n^{-1}V_n^{1/\beta}\right)_n$. 
\begin{lem} [Lemma 6 in \cite{KestenSpitzer}, Lemma 5 in \cite{FFN}]\label{Vn}
When $\alpha>1$, the sequence of random variables $\left(b_n^{-1}V_n^{1/\beta}
\right)_n$
converges in distribution to $|L|_\beta=\left( \int_{\mathbb R} L_1^\beta(x)\,
dx\right)^{1/\beta} $. 

When $\alpha=1$, the sequence of random variables $\left(b_n^{-1} V_n^{1/\beta}
\right)_n$
converges almost surely to $c^{\frac 1\beta}$. 
\end{lem}
\begin{lem}[Lemma 11 in \cite{BFFN}, Lemma 16 in \cite{FFN}]\label{Vn2}
If $\beta >1$, then 
$$\sup_n \mathbb{E}\left[ \left( \frac{b_n} {V_n^{1/\beta}}\right)^{\beta/(\beta-1)} \right] <+\infty.$$
If $\beta\le 1$, then  for every $p\geq 1$, 
$$\sup_n \mathbb{E}\left[ \left( \frac{b_n} {V_n^{1/\beta}}\right)^{p} \right] <+\infty.$$
\end{lem}
\noindent The idea will be that $V_n$ is of order $b_n^{\beta}$.
Therefore the study of $\sum_{n\ge 1}e^{-|tb_n|^\beta(A_1+iA_2\, sgn(t)) }$
will be useful in the study of $\psi(t)$. 
For any function $g: \RR_+\rightarrow \RR$,  we denote by ${\mathcal L}(g)$ the Laplace transform of 
$g$ given, for every $z\in\CC$ with $\mbox{\rm Re}(z)>0$, by
$${\mathcal L}(g) (z) =\int_{0}^{+\infty} e^{-z t} g(t) dt,$$
when it is well defined.
\begin{lem}\label{kaa}
When $\alpha>1$,  for every complex number $z$ such that
$\text{\rm Re}(z)>0$ and every $p\ge 0$, we have
$$ K_{p,\alpha}(z):= \sup_{u>0}\left\vert \sum_{n\ge 1} e^{- zu^{\delta\beta}b_n^\beta}
(u^{\delta\beta}b_n^\beta)^{p}  -  \frac{1}{u\delta \beta} \Gamma(p+\frac{1}{\delta\beta}) z^{-(p+\frac{1}{\delta\beta})} \right\vert <+\infty.$$
When  $\alpha=1$, for every complex number $z$ such that
$\text{\rm Re}(z)>0$ and every $p\ge 0$, we have
$$ K_{p,1}(z):= \sup_{u>0}\left\vert \sum_{n\ge 1} e^{- z u b_n^\beta}
(u b_n^\beta)^{p}  -  u^p \mathcal{L}(\tilde{w}_p) (zu) \right\vert <+\infty,$$
where $\tilde w_p(t):=\tilde w_0(t) t^p$
with
$$\tilde w_0(t) :=\frac{(\beta-1)^{1-\beta} w\left(\frac{t^{\frac 1{\beta-1}}}
  {\beta-1}\right)^{2-\beta}}{1+w\left(\frac{t^{\frac 1{\beta-1}}}
  {\beta-1}\right)} \ \ \mbox{if}\ \ \beta>1$$
and 
$$\tilde w_0(t):={\mathbf 1}_{[(e/(1-\beta))^{1-\beta};+\infty)}(t)
   \frac{\Delta((1-\beta)t^{\frac 1{1-\beta}})
    ^{2-\beta}(1-\beta)^{1-\beta}}
  {\Delta((1-\beta)t^{\frac 1{1-\beta}})-1}\ \ \mbox{if}\ \ \beta<1.$$ 
Here $w$ is the Lambert function defined on $[0;+\infty)$ as the
inverse function of $y\mapsto y e^y$ (defined on $[0;+\infty)$)
and $\Delta$ is the function defined on $[e;+\infty)$ as the inverse function 
of $y\mapsto e^y/y$ defined
on $[1;+\infty)$.
\end{lem}
\begin{proof}
\underline{First, we consider the case when $\alpha>1$}. 
With the change of variable $y=(ux)^{\delta\beta}$, we get
$$\int_{0}^{+\infty}e^{-(ux)^{\delta\beta}z}(ux)^{p\delta\beta}\, dx = 
    \frac {1} {u\delta\beta } \int_0^{\infty} e^{-yz} y^{\frac{1}{\delta\beta}
      +p -1} dy =\frac 1{u\delta\beta} \Gamma \left(p+\frac{1}{\delta\beta}\right) 
   z^{-(p+\frac{1}{\delta\beta})}.$$
Let $\floor{x}$ denote the integer part of $x$. Observe that
$$ \sum_{n\ge 1} e^{-(un)^{\delta\beta}z}(un)^{p\delta\beta}=\int_{1}^{+\infty}
     e^{-(u\floor{x})^{\delta\beta}z}(u\lfloor x\rfloor)^{p\delta\beta}\, dx.$$
Let us write
$$E_p(u,x)=\left| e^{-(u\floor{x})^{\delta\beta}z} (u\lfloor x\rfloor)^{p\delta\beta}
   -e^{-(ux)^{\delta\beta}z}(ux)^{p\delta\beta}\right|.$$
By applying Taylor's inequality to the function $v\mapsto e^{-vz} v^p$
 on the interval $[(u\floor{x})^{\delta\beta},(ux)^{\delta\beta}]$, 
we obtain for every $x>2$ (use $\floor{x} \geq x/2$) 
$$E_p(u,x) \le
 (1+|z|)(1+p)(1+(ux)^{p\delta\beta}) e^{-(ux/2)^{\delta\beta}\text{\rm Re}(z)}
   \, u^{\delta\beta}\big(x^{\delta\beta} - \floor{x}^{\delta\beta}\big).$$
Next, by applying Taylor's inequality to the function $t\mapsto t^{\delta\beta}$ 
according that $\delta\beta>1$ or $\delta\beta<1$ 
(again use $\floor{x} \geq x/2$ in the last case), we have 
$$ E_p(u,x) \le (1+|z|) (1+p)(1+(ux)^{p\delta\beta})
   \max\big(1,2^{1-\delta\beta}\big)\, \delta\beta\, u^{\delta\beta}\, 
   e^{-(ux/2)^{\delta\beta} \text{\rm Re}(z)}\, x^{\delta\beta-1}.$$
Therefore, with the change of variable $t=(ux/2)^{\delta\beta}$,
we get
$$\int_2^{+\infty} E_p(u,x) \, dx
    \le (1+|z|) (1+p)
    \max\big(2^{\delta\beta} , 2\big) \int_0^{+\infty}(1+2^{p\delta\beta}t^p) 
    e^{-\text{\rm Re}(z) t}\, dt .$$

\underline{Now, we suppose that $\alpha=1$} and follow the same scheme. We observe that $\delta\beta=1$.
With the change of variable $t=x(\log x)^{\beta-1}$, we get
$$\int_{1}^{+\infty}e^{-zux(\log x)^{\beta-1}}(ux(\log x)^{\beta-1})^p\, dx=u^p{\mathcal L}(\tilde w_p)(zu). $$
Indeed, if $\beta>1$, 
we have $t=[(\beta-1)x^{\frac 1{\beta-1}}(\log (x^{\frac 1{\beta-1}}))]^{\beta-1}$
and so $x=\left[\exp\left(w\left(\frac{t^{\frac 1{\beta-1}}}{\beta-1}\right)\right)\right]^{\beta-1}$ which gives
$dx=\tilde w_0(t)\, dt$
(since $w'(y)=\frac{w(y)}{y(1+w(y))}$ and since $e^{w(x)}=\frac{x}{w(x)}$).
Moreover, if $\beta<1$, we have
$$t=\left[\frac{x^{\frac 1{1-\beta}}}  {(1-\beta)\log (x^{\frac 1{1-\beta}})}\right]^{1-\beta}
\ \mbox{and so}\ x=\left[\exp\left((1-\beta)\Delta\left((1-\beta)t^{\frac 1{1-\beta}}
\right)\right)\right],$$
which gives
$dx=\tilde w_0(t)\, dt$
(since $\Delta'(y)=\frac{\Delta(y)}{y(\Delta(y)-1)}$ and since $e^{\Delta(y)}=y\Delta(y)$).

Let $x_0:=\max(4,e^{2(1-\beta)})$.
We have
$$\left|\sum_{n\ge 1}e^{-zun(\log n)^{\beta-1}}(un(\log n)^{\beta-1})^p-
    u^p{\mathcal L}(\tilde w_p)(zu)\right|
   \le 2x_0\sup_y(e^{-Re(z)y}y^p)+\int_{x_0}^{+\infty}E_p(u,x)\, dx, $$
where $$E_p(u,x):=\left| e^{-(zu\floor{x})(\log \floor{x})^{\beta-1}}
      (u\floor{x}(\log \floor{x})^{\beta-1})^p
      -e^{-zux (\log x)^{\beta-1}} (ux(\log x)^{\beta-1})^p
    \right|.$$
Applying Taylor's inequality, 
we get
$$E_p(u,x) \le (1+|z|)(1+p)e^{-\frac{\text{\rm Re}(z) ux}4(\log x)^{\beta-1}}
  2^p(1+(2ux(\log x)^{\beta-1})^{p}) 8u
   (\beta-1 +\log(x))(\log x)^{\beta-2}$$
(using the fact that $\floor{x}\ge x/2$ and $\log(x)\ge\log \floor{x}\ge 
(\log x)/2\ge 1-\beta$ if $x\ge x_0$).
Hence there exists some $c$ depending only on $p$ and $|z|$
such that $\int_{x_0}^{+\infty} E_p(u,x)\, dx$ is less than
$$  c \int_{x_0}^{+\infty} (1+(2ux(\log x)^{\beta-1})^{p})e^{-\frac{\text{\rm Re}(z)ux}4(\log x)^{\beta-1}} u
   (\beta-1+\log(x))(\log x)^{\beta-2}\, dx.$$
With the change of variable 
$t=ux(\log x)^{\beta-1}$, for which we have $dt=u(\beta-1 +\log x)(\log x)^{\beta-2}\, dx$,
we get
$$ \int_{x_0}^{+\infty} E_p(u,x)\, dx
\le  c \int_{0}^{+\infty} (1+(2t)^{p})e^{-\frac{\text{\rm Re}(z) t}4}\, dt.$$
The last integral is finite.
\end{proof}
%
\begin{lem}\label{controlL}
For every complex number $z$ such that
$\text{\rm Re}(z)>0$ and every $p\ge 0$,  we have
$$\lim_{u\rightarrow 0^{+}} [ (z u)^{p+1}  {\mathcal L}(\tilde w_p)(z u)- \Gamma(p+1) (-\log |u|)^{1-\beta} ]= 0$$
  and for every $u_0>0$,
  $$ \sup_{0<u<u_0} \frac{u^{p+1}{\mathcal L}(\tilde w_p)(u)}{
    \Gamma(p+1) (- \log |u|)^{1-\beta} }<\infty.$$
Hence, if $\alpha=1$, for every $z$ such that $\textrm{Re}(z)>0$, we have
$$\sum_{n\ge 1}e^{-zu^\beta b_n^\beta}\left(u^\beta b_n^\beta\right)^p\sim_{u\rightarrow 0+}
     \Gamma(p+1)\frac{u^{-\beta}(-\log(|u|^\beta))^{1-\beta}}{z^{p+1}}. $$
\end{lem}
\begin{proof}
We know that $w(x)\sim_{+\infty}\log x$ and $\Delta(x)\sim_{+\infty} \log x$. 
Hence, for every $p\geq 0$, we have 
$$\tilde w_p(t)\sim_{t \rightarrow +\infty} t^p (\log t)^{1-\beta}.$$ 
Now, we apply Tauberian theorems
(Theorems p. 443--446 in \cite{Fel}) to  Laplace transforms ${\mathcal L} (.)$ defined for complex numbers such that 
$\text{\rm Re}(z)>0$.  The lemma follows.
\end{proof}

\begin{lem}
There exist a sequence of random variables $(a_n)_n$ and a random variable
$A$ defined on $(0,1)$ endowed with the Lebesgue measure $\lambda$ such that, for every
$n\ge 1$, $a_n$ and $ (b_nV_n^{-\frac 1\beta})^{\frac 1{\delta}}$ have the same 
distribution, 
such that ${\mathbb E}_\lambda[\sup_{n\ge 1} a_n]<+\infty$
and such that $(a_n)_n$ converges almost surely and in $L^1$ to the random variable $A$.
\end{lem}
\begin{proof}
Lemmas   \ref{Vn} and \ref{Vn2}  insure in particular the uniform integrability of 
$\left( (b_n V_n^{-\frac 1{\beta}})^{\frac 1\delta}
\right)_n$ and the existence of
a sequence of random variables $(a_n)_n$ and of a random variable
$A$ defined on $(0,1)$ endowed with the Lebesgue measure $\lambda$ such that, for every
$n\ge 1$, $a_n$ and $ (b_n V_n^{-\frac 1\beta})^{\frac 1{\delta}}$ have the same 
distribution, such that $A$ has the same distribution
as  $|L|_\beta^{-1/\delta}$ if $\alpha>1$ and is equal to $c^{-1}$ if $\alpha=1$,
such that ${\mathbb  E}_\lambda[\sup_{n\ge 1}a_n]<+\infty$
and $(a_n)_n$ converges almost surely to $A$ and so in $L^1$. Indeed, following the
the Skorohod representation theorem, we define
$$a_n(x):=\inf\left\{u>0\ :\ {\mathbb P}\left( (b_nV_n^{-\frac 1{\beta}})^{\frac 1\delta}
  \le u\right)\ge x\right\}$$
and $A$ as follows~:
$$A(x):= c^{-1}\ \ \mbox{if}\ \alpha=1$$
and
$$A(x):= \inf\left\{u>0\ :\ {\mathbb P}\left( |L|_\beta^{-1/\delta}  \le u\right)\ge x\right\}
\ \ \mbox{if}\ \alpha>1.$$
The sequence $(a_n)_n$ converges almost surely to the random variable $A$ as $n$ goes to infinity. 
Moreover, from the formula
${\mathbb E}_\lambda[\sup_{n\ge 1}a_n]=\int_0^{+\infty}\lambda
(\sup_na_n>t) \, dt$
and the fact that
$$\sup_n a_n(x)>t\ \Leftrightarrow\ \inf_n {\mathbb P}\left( 
(b_n V_n^{-\frac 1\beta})^{\frac 1{\delta}}
\le t\right) <x,
$$
we get
$$\lambda(\sup_na_n>t) = \sup_n  {\mathbb P}\left( 
(b_nV_n^{-\frac 1 {\beta}})^{\frac 1\delta}
> t\right)\le t^{-\gamma}\sup_n
   {\mathbb E}\left[ b_n^{\frac\gamma\delta}V_n^{-\frac{\gamma}{\delta\beta}}
\right] ,$$
with $\gamma:=2$ when $\beta\le 1$ ;
$\gamma:=\delta\beta/(\beta-1)$ when $\alpha>1$ and $\beta >1$ or when $\alpha=1$ 
and $\beta\in(1,2)$. In each of this case, this gives 
${\mathbb E}_\lambda[\sup_{n\ge 1}a_n]<+\infty$ from Lemma \ref{Vn2}.
If $\alpha=1$ and $\beta=2$, we take $\gamma=2$ and use the fact that
$$ \sup_n{\mathbb E}\left[ b_n^{\frac\gamma\delta}V_n^{-\frac{\gamma}{\delta\beta}}
\right] = \sup_n {\mathbb E}\left[ \left(\frac{n\log n}{V_n}\right)^2
\right]\le \sup_n\frac {(\log n)^2}{n^2}{\mathbb E}[R_n^2]<\infty,$$
with $R_n:=\#\{y\in{\mathbb Z}\ :\ N_n(y)>0\}$ (according to 
inequality (3b) in \cite{LGR}) and we conclude analogously.
\end{proof}

Therefore,  using the previous lemma, the series $\psi$ can be rewritten, for
every real number $t\neq 0$, as 
\begin{equation}\label{EQ1}
\psi(t) ={\mathbb E}_\lambda\left[\sum_{n\ge 1}e^{-{|t|}^{\beta} 
    b_n^{\beta}a_ n^{-\delta\beta}(A_1+iA_2 \mbox{\small sgn}(t))}\right].
\end{equation}
\begin{lem}\label{UI}
There exists $t_0>0$ such that when $\alpha>1$ or $(\alpha=1,  \beta> 1)$, 
the family of random variables
$$ \left(\frac 1{\gamma(t)}\sum_{n\ge 1}e^{-{|t|}^{\beta} 
    b_n^{\beta}a_ n^{-\delta\beta}(A_1+iA_2 \mbox{\small sgn}(t))}\right)_{0<|t|<t_0}$$
is uniformly integrable and such that, if $\alpha>1$, the family
$$\left(\frac {|t|^\beta}{\gamma(t)}
\sum_{n\ge 1}  b_n^{\beta}a_ n^{-\delta\beta} e^{-|t|^\beta (a_n^{-1} n)^{\delta\beta}
(A_1+iA_2 \text{sgn} (t) )} \right)_{0<|t|<t_0} $$
is also uniformly integrable.
\end{lem}
\begin{proof}
If $\alpha>1$, thanks to lemma \ref{kaa}, we know that, 
for every real number $t\in(0,1)$ and every complex number $z$ such that $\text{\rm Re}(z)>0$,
we have
%
\begin{equation}\label{eq1}
\left|  \sum_{n\ge 1}e^{-|t|^\beta (a_n^{-1} n)^{\delta\beta}z}\right|
  \le  |t|^{-\frac 1\delta} 
    \frac{\sup_n a_n}{\delta\beta}\Gamma\left(\frac 
    1{\delta\beta}\right)|z|^{-\frac 1{\delta\beta}}+K_{0,\alpha}(z),
\end{equation}
and,
\begin{equation}\label{eq2}
|t|^\beta  \left|  \sum_{n\ge 1} (a_n^{-1} n)^{\delta\beta}   
   e^{-|t|^{\beta}(a_n^{-1} n)^{\delta\beta}z } \right|
  \le  |t|^{-\frac 1\delta} 
    \frac{\sup_n a_n}{\delta\beta} \Gamma\left(1+\frac1{\delta\beta}\right) 
   |z|^{-(1+\frac 1{\delta\beta})}+K_{1,\alpha}(z),
\end{equation}
from which we conclude.

Now, let us consider the case $\alpha=1$ and $\beta > 1$.
According to lemmas \ref{kaa} and \ref{controlL}, since
$0<A_1{|t|}^{\beta}(1+\sup_n a_n)^{-1}\le A_1$, we have
\begin{eqnarray*}
 \left|\sum_{n\ge 1} e^{-{|t|}^{\beta}  b_n^{\beta}a_n^{-1}(A_1+iA_2 \mbox{\small sgn}(t))}\right|
& \le  &K_{0,1}(A_1)+ {\mathcal L}(\tilde w_0)\left( A_1{|t|}^{\beta}(1+\sup_n a_n)^{-1} \right)\\
& \le &K_{0,1}(A_1) + c_0  |t|^{-\beta}\  (1+\sup_n a_n)  \left( - \log  \Big( A_1{|t|}^{\beta}(1+\sup_n a_n)^{-1} \Big)\right)^{1-\beta}
\end{eqnarray*}
for some positive constant $c_0>0$. Hence, there exists $t_1\in(0,1)$ 
such that for every $0<|t|<t_1$, 
$$
\frac{1}{|\gamma(t)|} \left|\sum_{n\ge 1} e^{-{|t|}^{\beta}     b_n^{\beta}a_n^{-1}(A_1+iA_2 \mbox{\small sgn}(t))}\right|
 \le c_1 \left[1+ (1+\sup_n a_n)  \left( \frac{- \log  \Big( A_1{|t|}^{\beta}(1+\sup_n a_n)^{-1} \Big)}{- \log  \Big( A_1{|t|}^{\beta}c\Big)}\right)^{1-\beta}\right]$$
for some positive constant $c_1>0$.\\*
\noindent If $\beta> 1$, since $c(1+\sup_n(a_n) )\ge 1$ a.s., the right-hand side of the above inequality is almost surely less than
$c_1 (2+\sup_n a_n)$ for every $|t|<(cA_1)^{-\frac 1\beta}$. Then, we can choose $t_0$ as the infimum of $(cA_1)^{-\frac 1\beta} $ and $t_1$.
The uniform integrability then follows from Lemma 9. 
\end{proof}
\begin{lem}\label{cvps}
If ($\alpha> 1,\beta\ne 1$) or $(\alpha=1,\beta >1)$,  we have
$$\lim_{t\rightarrow 0}\frac 1{\gamma(t)}
  \sum_{n\ge 1} \left(e^{-|t|^\beta(a_n^{-\delta\beta}b_n^\beta) 
  (A_1+iA_2 \text{sgn} (t))}-
   e^{-|t|^\beta(A^{-\delta\beta} b_n^{\beta})
   ( A_1+i A_2 \text{sgn} (t))  }\right)=0\, \, a.s..$$
Moreover, when $\alpha >1$ and $\beta <1$ , we have 
   $$\lim_{t\rightarrow 0} \frac{|t|^{\beta} } {\gamma(t)}
  \sum_{n\ge 1} \left( (a_n^{-\delta\beta}b_n^\beta)\ e^{-|t|^\beta(a_n^{-\delta\beta}b_n^\beta) 
  (A_1+iA_2 \text{sgn} (t))}- (A^{-\delta\beta} b_n^\beta)    e^{-|t|^\beta(A^{-\delta\beta} b_n^{\beta})
   ( A_1+i A_2 \text{sgn} (t))  }\right)=0\, \, a.s..$$
\end{lem}
\begin{proof}
We only prove the first assertion, the proof of the second one following the same scheme. Let $\beta\neq 1$ and $\alpha\geq 1$. 
Let $\varepsilon \in (0, 1/\delta)$,
$$\sum_{n=1}^{\floor{|t|^{\varepsilon-1/\delta}} } 
  \left|\left(e^{-|t|^\beta a_n^{-\delta\beta}b_n^\beta(A_1+ i A_2 \text{sgn} (t))} -
 e^{-|t|^\beta A^{-\delta\beta} b_n^\beta  (A_1+ i A_2\text{sgn} (t))}\right)\right|
  =\mathcal{O}(|t|^{\varepsilon-1/\delta}) ={o}(\gamma(t)).$$ 
Now it remains to prove 
the almost sure convergence to 0 as $t$ goes to $0$ of the following quantity~:
 $$\varepsilon_t:= \frac 1{\gamma(t)}    \sum_{ n > [|t|^{\varepsilon-1/\delta}] } 
   \left(e^{-|t|^\beta a_n^{-\delta\beta} b_n^{\beta}(A_1+iA_2\text{sgn} (t))}-
 e^{-|t|^\beta A^{-\delta\beta} b_n^{\beta}(A_1+iA_2\text{sgn} (t) )}\right).$$ 
By applying Taylor's inequality to the function 
$v\mapsto e^{-|t|^\beta|v|^{\delta\beta} b_n^\beta(A_1+iA_2\text{sgn}(t))}$,  
we have
 \begin{eqnarray*}
|\varepsilon_t|
&\leq &\delta\beta (A_1+|A_2|) \frac{ |t|^\beta}{\gamma(t)}  
\sum_{ n > [|t|^{\varepsilon-1/\delta}] }   (\inf_{n>[|t|^{\varepsilon-1/\delta}]}
      a_n)^{-\delta\beta} b_n ^\beta  e^{-A_1 |t|^\beta (\sup_{n>[|t|^{\varepsilon-1/\delta}]}
  a_n)^{-\delta\beta} b_n^\beta } \times\\
&\, &\ \ \ \ \ \ \ \ \ \ \ \ \ \ \ \ \ \ 
\ \ \ \ \ \ \ \ \ \ \ \ \ \ \ \ \ \ 
  \times \left|\frac{a_n^{-1}-A^{-1}} {(\sup_{n>[|t|^{\varepsilon-1/\delta}]} a_n)^{-1}}\right|\\ 
 &=&o(1) \ \ \mbox{\rm a.s.},
\end{eqnarray*} 
using lemmas \ref{kaa} and \ref{controlL} and
according to the fact that $(a_n)_n$ converges almost surely to $A$.
\end{proof}
\begin{proof}[Proof of Proposition \ref{series-2}]
First consider the case $\alpha>1$ and $\beta\neq1$. Thanks to 
lemmas \ref{kaa} and \ref{cvps}, we get that  
$$\frac 1{\gamma(t)}\left[\sum_{n\ge 1}  e^{-|t|^\beta a_n^{-\delta\beta}b_n^\beta(A_1+iA_2 
\text{sgn} (t))} -  \frac{A}{\delta\beta} \Gamma(\frac{1}{\delta\beta}) (A_1+iA_2 \text{sgn} (t))^{-1/\delta\beta} |t|^{-1/\delta} \right]   
\rightarrow 0\ \mbox{\rm a.s.}$$
 as $t$ goes to $0$.
Therefore, thanks to (\ref{EQ1}) and to the uniform integrability (Lemma \ref{UI}), we  
deduce $(\ref{lim-psi})$. The proof of $(\ref{lim-psi2})$ is similar (using Lemma \ref{controlL}) 
and is omitted. \\*
Again, to prove (\ref{lim-psi-deri}), we use (\ref{EQ1}). Since for $t\neq 0$,
$$\psi'(t)=- \beta\text{sgn}(t)  (A_1 +i A_2\text{sgn}(t)   ) |t|^{\beta-1}  \sum_{n\ge 1} 
 {\mathbb E}\left[ a_n^{-\delta\beta} b_n^\beta e^{-|t|^\beta 
    (a_n^{-\delta\beta}b_n^\beta (A_1+i A_2\text{sgn}(t) )} \right],$$ 
and
$$\gamma'(t)=- \frac{C}{\delta}  (A_1 +i A_2\text{sgn}(t))^{-1/\delta\beta} |t|^{-1/\delta-1},$$
  we decompose 
  $$\left(\frac{C}{\delta\beta}  (A_1 +i A_2\text{sgn}(t))^{-(1+1/\delta\beta)}\right)\left[\frac {\psi'(t)}{\gamma'(t)}-1\right]$$ as the sum of 
$$|t|^{1/\delta}  
 {\mathbb E}\left[ \sum_{n\ge 1} \left(e^{-|t|^\beta(a_n^{-\delta\beta}
      b_n^\beta (A_1+iA_2 \text{sgn} (t))}(|t|^{1/\delta} a_n^{-1})^{\delta\beta}b_n^{\beta}
-  e^{-|t|^\beta A^{-\delta\beta}b_n^{\beta} (A_1+iA_2 \text{sgn} (t))}  
   (|t|^{1/\delta} A^{-1})^{\delta\beta} b_n^{\beta}\right) \right]$$
   and
$$|t|^{1/\delta}{\mathbb E}\left[ \sum_{n\ge 1}  e^{-|t|^\beta A^{-\delta\beta}b_n^\beta (A_1+iA_2 \text{sgn} (t))}  
   (|t|^{1/\delta} A^{-1})^{\delta\beta} b_n^{\beta} - \frac{A}{|t|^{1/\delta}\delta\beta} \Gamma(1+\frac{1}{\delta\beta})(A_1 +i A_2\text{sgn}(t))^{-(1+1/\delta\beta)} \right]$$
The second assertion in Lemma \ref{cvps} and the uniform integrability in Lemma \ref{UI} implies that the first sum goes to 0 as $t$ goes to 0.
From Lemma \ref{kaa}, we get that the second one goes to 0 as  $t$ goes to 0. 
\end{proof}
%
%
%
%
%
%
%
\section{Proof of Theorem \ref{threc2} }
\noindent We first begin to prove that for every $a\in \mathbb{R}$,  the sequence of
$$K_{n,a}(h)= \sum_{k=1}^n\{  \mathbb{E} [h(Z_k)] - \mathbb{E} [h(Z_k-a)]\} $$ 
converges as $n$ tends to infinity.
Indeed, for every $a\in\mathbb{R}$, we have
\begin{equation}\label{int-1} 
K_{n,a}(h)= \frac{1}{2\pi} \int_{\mathcal I} \hat{h} (t) \left(\sum_{k=1}^n \mathbb{E} [e^{i t Z_k} ] \right) \left( 1- e^{-ita} \right) dt.
\end{equation}

\begin{prop}\label{tildepsi}
\begin{itemize}
\item[i)-] The series 
$$ \sum_{n\ge 1} | {\mathbb E}[e^{itZ_n}] | $$
is bounded on any set $[r,+\infty[$ with $r>0$ and so the series 
$$\tilde{\psi}(t):= \sum_{n\ge 1} {\mathbb E}[e^{itZ_n}] $$
is well defined for every  $t\neq 0$.
\item[ii)-] We have
\begin{equation*}
\lim_{t\rightarrow 0} \frac 1{\gamma(t)} \sum_{n\ge 1} \left| {\mathbb E}[e^{itZ_n}] -  {\mathbb E}\left[e^{-|t|^\beta V_n(A_1+iA_2 sgn(t))}\right]\right|=0,
\end{equation*}
and so
\begin{equation*}
\lim_{t\rightarrow 0} \frac 1{\gamma(t)} [\tilde{\psi}(t)-\psi(t)]=0,
\end{equation*}\end{itemize}
\end{prop}
\begin{proof}
In  order to prove {\it ii)}, we  show that 
$$\lim_{t\rightarrow 0}
 \frac 1{\gamma(t)}
\sum_{n\ge 1} \left| {\mathbb E}\left[\prod_{y\in\mathbb Z}\varphi_\xi(tN_n(y))\right]
 - {\mathbb E}\left[e^{-|t|^\beta V_n(A_1+iA_2 sgn(t))}\right]\right|=0.$$
{}From Lemma 6 in \cite{BFFN} and Lemma 12  in \cite{FFN}, 
for every $\eta>0$ and every $n\ge 1$, there exists a subset 
$\Omega_n$ such that  
for every $p> 1$, $\mathbb{P}(\Omega_n)=1- o(n^{-p})$ and such that, on $\Omega_n$,
we have 
$$N_n^{*} = \sup_x N_n(x) \leq n^{1-\frac 1\alpha +\eta }\ \ \ \mbox{and}\ \ \ 
V_n \geq n^{\delta\beta - \eta'},$$
with $\eta'=\frac{\eta\beta}2$ if $\alpha>1,\beta>1$ ; $\eta'=\eta(1-\beta)$ if 
$\alpha>1,\beta\le 1$ and $\eta'=\eta(1-\beta)_+$ if $\alpha=1$.
Hence, it is enough to prove that
$$\sum_{n\geq 1}\left| {\mathbb E}\left[E_n(t)
{\mathbf 1}_{\Omega_n}  \right]\right| = o(\gamma(t)) \ \ \mbox{as}\ \ t\rightarrow 0,$$
with $E_n(t):=\prod_{y\in\mathbb Z}\varphi_\xi(tN_n(y)) - 
e^{-|t|^\beta V_n(A_1+iA_2 sgn(t))}$.

\noindent In \cite{BFFN,FFN}, we also define some $\bar\eta\le \eta\max(1,\beta^{-1})$ 
and we take some
$\eta$ such that $\eta+\bar\eta<\frac 1{\alpha\beta}$. Hence, for every $\varepsilon_0>0$, there exists $n_1$ such that for every $n\ge n_1$, we have
$n^{\eta+\bar\eta-\frac 1{\alpha\beta}} \le \varepsilon_0$. 

\noindent In the proofs of propositions 8, 9 and 10 of \cite{BFFN} 
(and propositions 14, 15 of \cite{FFN}) or using the strong lattice property, 
we prove that there exist
$c>0$, $\theta>0$ and $n_0$ such that for every $t$ and every integer $n\ge n_0$ and such that 
$|t|>n^{-\delta+\bar\eta}$, we have, on $\Omega_n$,
$$\prod_{y\in\mathbb Z}|\varphi_\xi(tN_n(y))|  \le  e^{-cn^\theta}\ \ 
\mbox{and}\ \ 
\prod_{y\in\mathbb Z}|\phi(tN_n(y))|  \le  e^{-cn^\theta}.$$
Now, let $t$ and $n\ge n_1$ be such that $|t|\le n^{-\delta+\bar\eta}$.
Recall that we have
$$\Big|\varphi_\xi(u)-\phi(u)\Big|\le |u|^\beta h(|u|),$$
with $h$ a continuous and monotone function on $[0;+\infty)$ vanishing in 0.
Therefore there exist $\varepsilon_0>0$ and $\sigma>0$ such that, for every
$u\in[-\varepsilon_0;\varepsilon_0]$, we have
$$\max(|\phi(u)|,|\varphi_\xi(u)|)\le \exp(-\sigma |u|^\beta).$$
We have
$$|E_n(t)| \le
 \sum_y  \left(\prod_{z<y} |\varphi_\xi(tN_n(z))|\right)
\left|\varphi_\xi(tN_n(y))- \phi(tN_n(y)) \right|\left(\prod_{z>y} |\phi(tN_n(z))| \right).
$$ 
Now, since $|t|\le n^{-\delta+\bar\eta}$, on $\Omega_n$, for every $y\in\mathbb Z$,
 we have $|t| N_n(y)\leq
  n^{\eta+\bar\eta-\frac 1{\alpha\beta}} \leq \varepsilon_0$, 
we get
\begin{eqnarray*} 
|E_n(t)|
&\le&  \sum_y   h(n^{\eta+\bar\eta-\frac 1{\alpha\beta}}) |t|^\beta
  N_n(y)^\beta  \exp(-\sigma|t|^\beta V_n) \exp(\sigma\varepsilon_0^{\beta})\\
&\le&   h(n^{\eta+\bar\eta-\frac 1{\alpha\beta}}) |t|^\beta
  V_n  \exp(-\sigma|t|^\beta V_n) \exp(\sigma\varepsilon_0^{\beta}).
\end{eqnarray*}

\noindent Now, we fix some $t\ne 0$.
Let us write 
$${\mathcal N}_1(t):=\{n\ge 1\ :\ n\ge n_0,\ |t|>n^{-\delta+\bar\eta}\}$$
and
$${\mathcal N}_2(t):=\{n\ge 1\ :\ n\ge n_1,\ |t|\leq n^{-\delta+\bar\eta},\  n>  
|t|^{-\frac 1{2\delta} }\}.$$
We have
$$\sum_{n\le \max(n_0,n_1)} |E_n(t)| \le 2 \max(n_0,n_1), $$
$$\sum_{n\le |t|^{-\frac 1{2\delta} } } |E_n(t)| \le 2 |t|^{-\frac 1{2\delta} } =
 o(\gamma(t)), \ \ \mbox{as }\ t\rightarrow 0, $$
$$\sum_{n\in{\mathcal N}_1(t)}|E_n(t)|\le
       2 \sum_{n\ge 1} e^{-cn^\theta}$$
and
\begin{eqnarray*}
\sum_{n\in{\mathcal N}_2(t)}  \mathbb{E} [|E_n(t)| \mbox{\bf 1}_{\Omega_n} ]  &\le &
    \exp(\sigma\varepsilon_0^{\beta})    h\left(t^{ -\frac 1{2\delta}\left(\eta+\bar\eta-\frac 1{\alpha\beta}  \right)}  \right) |t|^\beta
  \mathbb{E}  \Big[ \sum_{n\ge 1} V_n  \exp(-\sigma|t|^\beta V_n)\Big]  =  
o(\gamma(t)),
    \end{eqnarray*}
as $t\rightarrow 0$, using Proposition~\ref{series-2} and the continuity of 
the function $h$ at 0 (and the fact that 
$|t|^\beta V_n\exp(-\sigma|t|^\beta V_n)\le k_0 \exp(-\frac 1 2 \sigma|t|^\beta V_n)$
for some $k_0>0$). 
Then, ii)- is proved and i)- can easily be deduced 
from the above arguments.
\end{proof}

The integrand in (\ref{int-1}) is bounded by 
$\Theta(t):= |\hat{h} (t)|  |1-e^{-ita}|\ \sum_{n\geq 1}  | {\mathbb E}[e^{itZ_n}] | $.

Let $r>0$, on the set $\{t;  |t| \geq r\}$, by i)- from Proposition \ref{tildepsi}, since $\hat{h}$ is integrable, $\Theta$ is 
integrable. From Propositions~\ref{series-2} and \ref{tildepsi} (item ii)-) and from 
the fact that $\hat{h}$ is continuous at 0,  
$\Theta(t)$ is in $O\left(|t| \gamma(t)\right)$ (at $t=0$), which is integrable in the 
neighborhood of 0 in all cases considered in Theorem \ref{threc2} except $(\alpha,\beta)=(1,2)$. From the dominated convergence theorem, we deduce that 
\begin{equation}\label{ControleKna}
\lim_{n\rightarrow +\infty} K_{n,a}(h) =\frac{1}{2\pi} \int_{\mathcal I} 
\hat{h} (t) \tilde\psi(t) \left( 1- e^{-ita} \right) dt.
\end{equation}
In the case $(\alpha,\beta)=(1,2)$, by assumption, for every integer $n\ge 1$, 
the function $t\rightarrow \hat{h}(t) \sum_{k=1}^n \mathbb{E} [e^{i t Z_k} ] $  being even, we have
 \begin{equation}\label{int-2} 
K_{n,a}(h)= \frac{1}{2\pi} \int_{\mathcal I} \hat{h} (t) \left(\sum_{k=1}^n \mathbb{E} [e^{i t Z_k} ] \right) \left( 1- \cos(ta) \right) dt.
\end{equation}
The integrand in (\ref{int-2}) is uniformly bounded in $n$ by a function in $O\left( \log(1/|t| )^{-1} \right)$ (at $t=0$), which is integrable in the neighborhood of 0. From the dominated convergence theorem, we deduce that 
\begin{equation}\label{ControleKna2}
\lim_{n\rightarrow +\infty} K_{n,a}(h) =\frac{1}{2\pi} \int_{\mathcal I} 
\hat{h} (t) \tilde\psi(t) \left( 1- \cos(ta) \right) dt.
\end{equation}

{\it In the rest of the proof we only consider the strongly non-lattice case, the lattice case can  be handled in the same way}.

\noindent \underline{Let us first consider the case $\alpha>1,\beta\in(1,2]$}. 
We recall that, in this case, we have set 
$$C=(\delta \beta)^{-1} 
 \Gamma(\frac 1{\delta\beta}){\mathbb E}[|L|_\beta^{-\frac 1\delta}].$$
Since the function $t\rightarrow \hat{h}(t) \tilde\psi(t)$ is integrable on 
$\mathcal{I}\setminus[-\pi,\pi]$ (note that $\hat{h}$ is integrable and $\tilde\psi$ is bounded on   
$\mathcal{I}\setminus[-\pi,\pi]$ by Proposition \ref{tildepsi}), we have 
$$ \lim_{a\rightarrow +\infty} \frac{a^{1-1/\delta}}{2\pi} 
  \int_{\{|t|\geq \pi\}} \left| \hat{h} (t) \tilde\psi(t) 
\left( 1- e^{-ita} \right) \right| \, dt = 0.$$
We define the functions
\begin{equation}\label{fonctiong} 
g(t) := (1-e^{-it}) |t|^{-1/\delta}(A_1+iA_2\text{sgn}(t))^{-1/\delta\beta},\ \ 
g_{a}(t) := a g(at)
\end{equation}
and $f(t) := {\bf 1}_{[-\pi,\pi]} (t)\  \hat{h} (t) |t|^{1/\delta} 
\tilde\psi(t)(A_1+iA_2\text{sgn}(t))^{1/\delta\beta}$.   We have: 
$$
\frac{a^{1-1/\delta}}{2\pi} \int_{\{|t|\leq \pi\}} \hat{h} (t) \tilde\psi(t) \left( 1- e^{-ita} \right) dt 
= \frac{1}{2\pi}\int_{\mathbb{R}} f(t)\, g_{a}(t) dt = \frac{1}{2\pi}\, \big(f*g_{a}\big)(0). 
$$
Since $g$ is integrable on $\mathbb{R}$ and $f$ is bounded on $\mathbb{R}$ and continuous at $t=0$ with $f(0) = C\hat{h} (0)$ (by Propositions 4 and 12), it follows  
from classical arguments of approximate identity that 
$$\lim_{a\rightarrow +\infty} (f*g_{a}\big)(0) = C\hat{h} (0)\int_{\mathbb{R}} g(t)dt.$$  
Let us observe that 
$$\int_{\mathbb{R}} g(t) dt = 2 \textrm{Re}\left[(A_1+iA_2)^{-\frac 1{\delta\beta}}
     \int_0^{\infty} \frac{1-e^{-it}}{t^{1/\delta}}\, dt \right].$$
By applying the residue theorem to the function $z\mapsto z^{-1/\delta} 
(1-e^{-iz})$ with the contour in the complexe plane defined as follows~: 
the line segment from  $-ir$ to $-iR$ ($r<R$), the circular arc connecting $-iR$ to $R$, 
the line segment from $R$ to $r$ and the circular arc from $r$ to $-ir$ and letting 
$r\rightarrow 0, R\rightarrow +\infty$, we get that 
$$\int_0^{\infty} \frac{1-e^{-it}}{t^{1/\delta}}dt= \frac{\delta}{1-\delta} \Gamma\Big(2-\frac{1}{\delta}\Big)e^{i\frac{\pi}{2\delta}(1-\delta)}.$$ 
From this formula we easily deduce the first statement of theorem \ref{threc2}
using the fact that
$$(A_1+iA_2)^{-\frac 1{\delta\beta}}=\frac{e^{-i\frac\theta{\delta\beta}}}{(A_1^2+A_2^2)
   ^{\frac 1{2\delta\beta}}},\ \ \mbox{with}\ \ \theta=\arctan\left(\frac {A_2}{A_1}\right).$$
\underline{Now assume $\alpha\ge 1,\beta=1$ or $\alpha=1,\beta\in(1,2)$}. 
We have $\gamma(t)= b_t |t|^{-\beta}(-\log |t|)^{1-\beta}$
(with $b_t=A_1^{-1}$ if $\beta=1$ and with $b_t=c^{-1}(A_1+iA_2 sgn(t))^{-1}$ if $\alpha=1,\beta\in(1,2)$).
Moreover, by combining propositions \ref{series-2} and  \ref{tildepsi}, we have
$$ \lim_{t\rightarrow 0} \left| (\gamma(t))^{-1}  \tilde\psi(t) - 1\right| =0.$$ 
Hence, for every $\varepsilon\in (0,1)$, there exists $0<A_{\varepsilon}<1$ such that
\begin{equation}\label{estim}
\forall t,\ \ |t|\leq A_{\varepsilon}\ \Rightarrow\ [\ |\tilde\psi(t) -\gamma(t)|
   <\varepsilon\gamma(t)\ \mbox{and}\ |\hat{h}(t)-\hat{h}(0)|<\varepsilon\ ].
\end{equation}
Since  $\tilde\psi$ is bounded on $[A_{\varepsilon},+\infty[$ and $\hat{h}$ is integrable
on $\mathcal I$, we have
$$\left|\frac{1}{2\pi} \int_{t\in\mathcal I,\ |t|\geq A_{\varepsilon}} 
   \hat{h}(t) \tilde\psi(t) (1-e^{-ita})dt\right|\leq C(\varepsilon).$$
Let $a$ be such that $a\ge A_\varepsilon^{-1/\beta}$. We have
$$\left|\frac{1}{2\pi} \int_{\{|t| < a^{-\beta}\}} 
   \hat{h}(t) \tilde\psi(t) (1-e^{-ita})dt\right|\leq \frac{a}{\pi} || \hat{h}||_{\infty} 
\int_0^{a^{-\beta}} t\, |\gamma(t)|\, dt,$$
that can be neglected as $a$ goes to infinity since
$$\int_0^{a^{-\beta}} |at\gamma(t)|\, dt=\mathcal{O}\left( a^{(\beta-1)^2} 
         \log(a)^{1-\beta}\right)=o(  a^{\beta-1}  \log (a)^{1-\beta} )\ 
\mbox{as}\ a\rightarrow \infty\ 
        \ \mbox{if}\ \ \alpha=1,\  \beta\in (1,2)$$
as $a$ goes to infinity
and since
$$ \int_0^{a^{-\beta}} a t |\gamma(t)|\, dt=\mathcal{O}(1) = o(\log(a))\ \ 
\mbox{as}\ a\rightarrow \infty\ \mbox{if}\ \ \beta=1.$$
It remains to estimate  $\frac{1}{2\pi} 
\int_{\{a^{-\beta}\leq |t|\leq A_{\varepsilon}\}} \hat{h}(t) \tilde\psi(t) (1-e^{-ita})dt$ that we decompose into two parts:
\begin{equation*}
I_1(a):= \frac{1}{2\pi} \int_{\{a^{-\beta}\leq |t| < A_{\varepsilon}\} } 
[ \hat{h}(t) \tilde\psi(t) - \hat{h}(0)\gamma(t)]  (1-e^{-ita})  dt 
\end{equation*} 
and 
$$I_2(a):=  \frac{\hat{h}(0)}{2\pi} \int_{\{a^{-\beta}\leq |t| < A_{\varepsilon}\} }  (1-e^{-ita}) \gamma(t)  dt.  $$ 
\textbullet We first estimate $I_2(a)$ for $a$ large. Remark that  by the change of variables $u= at$,
$$I_2(a)=  \frac{\hat{h}(0) }{2\pi a} \int_{\{a^{1-\beta}<|u|<a A_{\varepsilon}\}}
  (1-e^{-iu})\gamma\Big(\frac{u}{a}\Big)\, du.  $$ 
We treat separately the cases $\beta=1$ and $\alpha=1,\beta\in (1,2)$. 
If $\beta=1$, we have
$$ \frac 1{2\pi a}\int_{\{ 1 <|u|< a A_{\varepsilon}\}}
  (1-e^{-iu})\gamma\Big(\frac{u}{a}\Big)\, du= \frac{1}{A_1 \pi}\int_{\{1<u< a A_{\varepsilon}\}}
  \frac{1-\cos u}{u}\, du\sim \frac 1{A_1 \pi}\log (a) $$
since
$$\lim_{x\rightarrow +\infty} \frac{1}{\log (x)} \int_{1}^{x} \frac{1-\cos(u)}{u} du=1.$$
This comes from the fact that $\left(\int_1^x\frac{\cos(t)}t\, dt\right)_x$
is bounded.

\noindent If $\alpha=1$ and $\beta\in(1,2)$, we have

$\displaystyle\frac 1{2\pi a}\int_{\{a^{1-\beta}<|u|<a A_{\varepsilon}\}}
  (1-e^{- iu})\gamma\Big(\frac{u}{a}\Big)\, du=$
\begin{eqnarray*}
&=& \frac {a^{\beta-1}\beta^{1-\beta}}{2\pi c}\int_{\{a^{1-\beta}<|u|<a A_{\varepsilon}\}}
  (1-e^{-iu}) 
(A_1+iA_2 \text{sgn}(u))^{-1}|u|^{-\beta}(\log(a)-\log |u|)^{1-\beta}\, du\\
&=& \frac { a^{\beta-1}(\log (a^\beta))^{1-\beta}}
   {2\pi c}\int_{\mathbb R} f_a(u)\, du,
\end{eqnarray*}
with 
$$f_a(u):={\mathbf 1}_{[a^{1-\beta},a A_{\varepsilon}]}(|u|)
    (1-e^{-iu})
   (A_1+iA_2 \text{sgn}(u))^{-1}|u|^{-\beta}\left(1-\frac{\log |u|}{\log a}\right)^{1-\beta}.$$
We observe that
$$|f_a(u)|\le  F(u):=\min(1,|u|)|A_1+iA_2|^{-1}|u|^{-\beta}\beta^{1-\beta}$$
(with $F$ integrable on $\mathbb R$ since $\beta\in(1,2)$)
and that we have
$$\forall u\ne 0,\ \lim_{a\rightarrow +\infty} f_a(u) =     (1-e^{-iu})
   (A_1+iA_2 \text{sgn}(u))^{-1}|u|^{-\beta}=:g(u).$$
So,
$$\lim_{a\rightarrow +\infty}
  \frac { (\log (a^\beta))^{\beta-1} } {2\pi a^{\beta}}\int_{\{a^{1-\beta}<|u|<a A_{\varepsilon}\}}
  (1-e^{-iu})\gamma(u/a)\, du=\frac {1}
   {2\pi c}\int_{\mathbb R}g(u)\, du.$$
We recall that
$$\int_{\mathbb{R}} g(t) dt = 2 \textrm{Re}\left[(A_1+iA_2)^{-1}\int_0^\infty
 \frac{1-e^{-it}}{t^\beta}\, dt\right]$$
and that
$$\int_0^{\infty} \frac{1-e^{-it}}{t^{\beta}}dt= 
 \frac {\Gamma(2-\beta)}{\beta-1} e^{\frac i 2(\beta-1)\pi}.$$
This gives
$$\lim_{a\rightarrow +\infty} \frac { (\log (a^\beta))^{\beta-1} } { a^{\beta-1}} I_2(a)=D_1.$$
\textbullet Second,  we estimate $I_1(a)$. From (\ref{estim}), we have 
 \begin{eqnarray*}
|I_1(a) | &\leq  & \frac{\varepsilon(\mathcal O(1)+|\hat{h}(0)|)}{\pi} 
  \int_{\{a^{-\beta}\leq |t| < A_{\varepsilon} \}} |1-e^{-ita}|  |\gamma(t)|\, dt \\
  &\leq & C  \frac{\varepsilon}{ a}\int_{\{a^{1-\beta}<|u|<a A_{\varepsilon}\}}
  |1-e^{-iu}|  \gamma\Big(\frac{u}{a}\Big)\, du.
  \end{eqnarray*} 
When $\beta=1$, $|I_1(a)| \leq \varepsilon \log(a)$. When $\alpha=1$ and $\beta\in (1,2)$,  from the above computations, we also have 
$$\lim_{a\rightarrow +\infty}
  \frac { (\log a)^{\beta-1} } {2\pi a^{\beta}}\int_{\{a^{1-\beta}<|u|<a A_{\varepsilon}\}}
  |1-e^{-iu}|  \gamma\Big(\frac{u}{a}\Big)\, du=\frac {1}
   {2\pi c}\int_{\mathbb R} |g(u) |\, du.$$
 Therefore, we get $| I_1(a)| \leq C  \varepsilon   a^{\beta-1} (\log a)^{1-\beta}.$
 
\noindent \underline{The case $(\alpha,\beta)=(1,2)$} can be handled in the same way as $\alpha=1, \ \beta\in(1,2)$ using the inequality 
$1-\cos(t)\leq \min (2, t^2 ).$ Details are omitted.

\section{Proof of Theorem~\ref{threc0} (transient case)}
We suppose that $\alpha>1$ and $\beta<1$. So $\delta>1$.
We will again use the notation 
$$C=(\delta \beta)^{-1} 
 \Gamma(\frac 1{\delta\beta}){\mathbb E}[|L|_\beta^{-\frac 1\delta}].$$
Let $h :{\mathbb R} \rightarrow {\mathbb C}$ be a Lebesgue-integrable function such that its Fourier transform 
$\hat h$ is differentiable, with $\hat h$ and $(\hat h)'$ Lebesgue-integrable. Then, using 
the Fourier inversion formula, we obtain for every $n\ge 1$,
$$2\pi{\mathbb E}[h(Z_n-a)] = 
   \int_{\mathbb R} \hat h(t){\mathbb E}[e^{itZ_n}]e^{-i ta}\, dt.$$
We get 
$$2\pi\sum_{n\ge 1}{\mathbb E}[h(Z_n-a)] = 
   \sum_{n\ge 1}\int_{\mathbb R} \hat h(t){\mathbb E}[e^{itZ_n}]  e^{-i ta}
  \, dt.$$
Since here $\beta<1$ (thus $\delta>1$), the function  $t\mapsto \hat{h}(t) \sum_{n\ge 1}\left| {\mathbb E} [e^{itZ_n}]\right|$ is integrable (note that $\sum_{n\ge 1}\left| {\mathbb E} [e^{itZ_n}]\right|$
corresponds to the case $A_2=0$, then use Proposition \ref{series-1-trans} and (\ref{lim-psi})). Therefore, from (\ref{def-psi}), we have
$$2\pi\sum_{n\ge 1}{\mathbb E}[h(Z_n-a)] = 
   \int_{\mathbb R} \hat h(t)\, \psi(t)\,  e^{-i ta}
  \, dt.$$
Let ${\mathcal S}({\mathbb R})$  denote the so-called Schwartz space. 
Let $r\in(0,+\infty)$ and let $\chi\in{\mathcal S}({\mathbb R})$ be such that 
\begin{equation}\label{cond-chi} 
|\chi|\leq 1\quad \text{ and } \quad \forall t\in[-r;r],\ \chi(t)=1.
\end{equation}
We have
$$2\pi\sum_{n\ge 1}{\mathbb E}[h(Z_n-a)] = I_1(a)+I_2(a)+I_3(a),$$
with
\begin{subequations}
\begin{eqnarray}
& & I_1(a):=   C\, \hat h(0) \int_{\mathbb R}\chi(t)\, |t|^{-\frac 1\delta}(A_1+iA_2\text{sgn}(t))^{-\frac1{\delta\beta}} \, e^{-i ta}
  \, dt,  \nonumber \\
& &  I_2(a):=   \int_{\mathbb R} \chi(t)\left\{\hat h(t)\psi(t) 
     - C\, \hat h(0) |t|^{-\frac 1\delta}(A_1+iA_2\text{sgn}(t))^{-\frac1{\delta\beta}} \right\} e^{-i ta}
  \, dt, \nonumber \\
& &  I_3(a):=   \int_{\{|t|>r\}}(1-\chi(t))\, \hat h(t)\, \psi(t)\, e^{-i ta}
  \, dt. \nonumber
\end{eqnarray}
\end{subequations}
The study of $I_3(a)$ is easy. Set $g_3(t) = (1-\chi(t))\hat h(t)\psi(t)$. 
From (\ref{cond-chi}), we have $I_3(a)= \widehat{g_3}(a)$, 
and from Propositions~\ref{series-1-trans} and 4, $g_3$ and $g_3'$ are Lebesgue-integrable on ${\mathbb R}$. An integration by parts then gives
$$I_3(a)=O(a^{-1}) = o(a^{1/\delta-1})\ \mbox{ as }a\mbox{ goes to }
\infty. $$
The next two subsections are devoted to the study of $I_1(a)$ and $I_2(a)$.  
\subsection{Study of $I_1(a)$} Let us prove that:  
$$ \lim_{a\rightarrow +\infty} a^{1/\delta-1}\, I_1(a) = C\, \hat h(0)\, c^-_{\delta,\beta}$$
where $c_{\delta,\beta}^-$ is a constant defined in Lemma~\ref{fourier-dist} below. The last property follows from Lemma~\ref{id-approche} below. Before let us establish the following.  
\begin{lem} \label{fourier-dist}
For every function $\displaystyle g\in{\mathcal S}({\mathbb R}),$
$$\ \int_{\mathbb R} \frac{\hat g(u)}{|u|^{1/\delta}(A_1+iA_2\ \text{\rm sgn}(u))^{\frac 1{\delta\beta}}}\, du = \int_{\mathbb R} \frac{g(v)}{|v|^{1-\frac 1\delta}} \left( c_{\delta,\beta}^{+} {\bf 1}_{\RR_+}(v) + c_{\delta,\beta}^{-} {\bf 1}_{\RR_-}(v)\right)\, dv,$$
where 
$$c_{\delta,\beta}^{+}:=  \frac{2\ \Gamma(1-\frac{1}{\delta})}{(A_1^2+A_2^2)^{\frac 1{2\delta\beta}}} \sin\left( \frac{1}{\delta}\Big(\frac{\pi}{2} + \frac{1}{\beta} \arctan 
\Big(\frac{A_2}{A_1}\Big)\Big) \right)$$
and
$$c_{\delta,\beta}^{-}:=  \frac{2\ \Gamma(1-\frac{1}{\delta})}{(A_1^2+A_2^2)^{\frac 1{2\delta\beta}}} \sin\left( \frac{1}{\delta}\Big(\frac{\pi}{2} - \frac{1}{\beta} \arctan 
\Big(\frac{A_2}{A_1}\Big)\Big) \right).$$
\end{lem}
\noindent Note that, since $\delta>1$, the functions $w\mapsto |w|^{-1/\delta}$ and $w\mapsto |w|^{-(1-\frac 1\delta)}$ are Lebesgue-integrable on any neighborhood of $w=0$, so that the two previous integrals  are well defined. 
\begin{proof}
For every $u\ne 0$, we have $$|u|^{-\frac 1\delta} (A_1+iA_2\text{sgn}(u))^{-\frac 1{\delta\beta}}= \int_0^{+\infty}
   e^{-x|u|^{\frac1{\delta}}(A_1+iA_2\text{sgn}(u))^{\frac 1{\delta\beta}}}\, dx.$$ 
For any $x>0$, let us denote by $f_x$ the Fourier transform of the function
$u\mapsto e^{-x|u|^{\frac 1\delta}(A_1+iA_2\text{sgn}(u))^{\frac 1{\delta\beta}}}$. 
By Fubini's theorem and Parseval's identity, we have
\begin{eqnarray*}
\int_{\mathbb R} \frac{\hat g(u)}{|u|^{\frac 1\delta}(A_1+iA_2\text{sgn}(u))^{\frac 1{\delta\beta}}}\, du
    &=&  \int_0^{+\infty}\left(\int_{\mathbb R} \hat g(u) e^{-x|u|^{\frac1{\delta}}(A_1+iA_2\text{sgn}(u))^{\frac 1{\delta\beta}}} \, du\right)dx \\
    &=& \int_0^{+\infty}\left(\int_{\mathbb R}g(v) f_x (v)\, dv\right) dx.
\end{eqnarray*}
Next, since we have: $\forall x>0,\ \forall v\in{\mathbb R},\ f_x(v)
        = x^{-\delta} f_1\left(\frac{v}{x^\delta}\right)$, 
we obtain, from Fubini's theorem, with the change of variable $y=|v|/x^\delta$ and
finally by the dominated convergence theorem (since $c_{\delta,\beta}^\pm$ are
well defined, see below), that
\begin{eqnarray*}
\int_{\mathbb R}\frac{\hat g(u)}{|u|^{\frac 1\delta}(A_1+iA_2\text{sgn}(u))^{\frac 1{\delta\beta}}}\, du
 &=& \lim_{A\rightarrow 0}\int_{\mathbb R} g(v) \left[\int_A^{+\infty} x^{-\delta} f_1
   \left(\frac{v}{x^\delta}\right)\, dx\right]\, dv\\
 &=& \lim_{A\rightarrow 0}\int_{\mathbb R} g(v) |v|^{1/\delta-1} 
     \left[\int_0^{|v|A^{-\delta}}\frac{f_1(\textrm{sgn}(v)y)}{\delta y^{1/\delta}}\, dy\right]\, dv\\
     &=& \int_{\mathbb R}\frac{g(v)}{|v|^{1-1/\delta}}\left(c_{\delta,\beta}^+{\bf 1}_{\RR_+}(v) + c_{\delta,\beta}^{-} {\bf 1}_{\RR_-}(v)\right) \, dv,
\end{eqnarray*}
with 
$$c_{\delta,\beta}^{+}:= \int_0^{+\infty}\frac{f_1(y)}{\delta y^{1/\delta}}\, dy\ \ \mbox{\rm and}\ 
\ c_{\delta,\beta}^{-}:= \int_0^{+\infty}\frac{f_1(-y)}{\delta y^{1/\delta}}\, dy.$$
Let us compute $c_{\delta,\beta}^{+}$. We have
\begin{eqnarray*}
c_{\delta,\beta}^+ &= &\lim_{A\rightarrow +\infty}\frac{1}{\delta} 
    \int_0^{A} f_1(y)\, y^{-1/\delta}\, dy\\
&=& \lim_{A\rightarrow +\infty} \frac{1}{\delta} \int_0^{A} y^{-1/\delta} \left(\int_{\mathbb R}  e^{ixy}   e^{-|x|^{\frac{1}{\delta}} (A_1+iA_2\text{sgn}(x))^{\frac{1}{\delta\beta}}} \, dx \right) dy \\
&=& \lim_{A\rightarrow +\infty} \int_{\mathbb R}|u|^{-\frac 1\delta}e^{iu}\left(
 \int_{\frac {|u|}A}^{+\infty} \frac{1}{\delta}v^{\frac 1\delta-1} 
   e^{-v^{\frac{1}{\delta}} (A_1+iA_2\text{sgn}(u))^{\frac{1}{\delta\beta}}} \, dv \right) du \\
&=& \lim_{A\rightarrow +\infty} \int_{\mathbb R}|u|^{-\frac 1\delta}e^{iu}
\frac{e^{-|u|^{\frac{1}{\delta}} A^{-\frac 1\delta}(A_1+iA_2\text{sgn}
(u))^{\frac{1}{\delta\beta}}}}{(A_1+iA_2\textrm{sgn}(u))^{\frac 1{\delta\beta}}}\, du\\
&=&\lim_{A\rightarrow +\infty} 2 \, \textrm{Re}\left[ 
\int_0^{+\infty}u^{-\frac 1\delta}e^{iu}
\frac{e^{-u^{\frac{1}{\delta}} A^{-\frac 1\delta}(A_1+iA_2
)^{\frac{1}{\delta\beta}}}}{(A_1+iA_2)^{\frac 1{\delta\beta}}}\, du\right],
\end{eqnarray*}
using the change of variables $(u,v)=(yx,x)$.
Now applying the residue theorem to the function $z\mapsto z^{-\frac 1\delta} 
e^{iz} e^{-z^{\frac 1\delta}A^{-\frac 1\delta} (A_1+iA_2)^{\frac{1}{\delta\beta}}}$ with the contour in the complexe plane defined as follows~: 
the line segment from  $r$ to $R$ ($r<R$), the circular arc connecting $R$ to $iR$, 
the line segment from $iR$ to $ir$ and the circular arc from $ir$ to $r$ and letting 
$r\rightarrow 0, R\rightarrow +\infty$, we get that 
$$
\int_0^{+\infty}u^{-\frac 1\delta}e^{iu}
{e^{-u^{\frac{1}{\delta}} A^{-\frac 1\delta}(A_1+iA_2)^{\frac{1}{\delta\beta}}}}\, du
= e^{i\left(\frac \pi 2-\frac \pi{2\delta}\right)}
  \int_0^{+\infty}t^{-\frac 1\delta}e^{-t}
{e^{-t^{\frac{1}{\delta}}e^{\frac {i\pi}{2\delta}} A^{-\frac 1\delta}(A_1+iA_2)^{\frac{1}{\delta\beta}}}}\, dt.
$$
Taking $A\rightarrow +\infty$, we get the expression of $c_{\delta,\beta}^+$.
\end{proof}
\begin{lem} \label{id-approche}
We have: $\displaystyle \lim_{a\rightarrow +\infty}a^{1-1/\delta} 
\int_{\mathbb R}\chi(t) |t|^{-\frac 1\delta} (A_1+iA_2\ \text{\rm sgn}(t))^{-\frac1{\delta\beta}} e^{-i ta}
  \, dt = c_{\delta,\beta}^-$. 
\end{lem}
\begin{proof}
Let $\gamma\in {\mathcal S}({\mathbb R})$ 
such that $\hat \gamma=\chi$, and define: $\forall x\in{\mathbb R},\ \tilde \gamma_{a}(x) := a \gamma(-ax)$.
{}From Lemma~\ref{fourier-dist} and from the change of variable $v=w a$, we get
\begin{eqnarray*}
\int_{\mathbb R} \chi(t) |t|^{-1/\delta}(A_1+iA_2\text{sgn}(t))^{-\frac 1{\delta\beta}}e^{-ita} \, dt 
&=& \int_{\mathbb R} \widehat{\gamma(\cdot+a)}(t)|t|^{-1/\delta}(A_1+iA_2\text{sgn}(t))^{-\frac 1{\delta\beta}}\, dt \\ 
&=&\int_{\mathbb R} \gamma(v+a)|v|^{1/\delta-1}(c_{\delta,\beta}^+{\bf 1}_{\RR_+}(v) + c_{\delta,\beta}^{-} {\bf 1}_{\RR_-}(v))\, dv\\
&=& a^{1/\delta-1}\int_{\mathbb R} a\,  \gamma\left(a\big(w+1\big)\right) g_\delta(w)\, dw\\
&=& a^{1/\delta-1}\int_{\mathbb R}\tilde \gamma_{a}\left(-1-w\right)\, g_\delta(w)\, 
dw\\
&=&  a^{1/\delta-1} \big(\tilde \gamma_{a}*g_\delta\big)(-1),
\end{eqnarray*}
where $*$ denotes the convolution product on ${\mathbb R}$ and $g_\delta(v):=|v|^{1/\delta-1}(c_{\delta,\beta}^+{\bf 1}_{\RR_+}(v) + c_{\delta,\beta}^{-} {\bf 1}_{\RR_-}(v)) $. Observe that we have 
$$\int_{\mathbb R}\tilde \gamma_{a}(w)dw = \int_{\mathbb R} \gamma(y)dy = \chi(0) = 1.$$
Now, from the fact that $\tilde \gamma\in{\mathcal S}({\mathbb R})$ (actually use 
$\sup_{x\in {\mathbb R}} (1+x^2)|\tilde \gamma(x)| < \infty$), that $g_\delta$ is continuous at $-1$ and that the function  $w\rightarrow w^{-2} g_\delta(w)$ is Lebesgue-integrable at infinity, it can be easily deduced from classical arguments of approximate identity that we have (see Prop.~1.14 in D.~Guibourg's thesis \cite{Gui} for details): $\lim_{a\rightarrow +\infty} (\tilde \gamma_{a}*g_\delta)(-1) = g_\delta(-1) = c_{\delta,\beta}^-$. 
\end{proof}
\subsection{Study of $I_2(a)$}
Let us prove that:  
$$I_2(a) = o(a^{1/\delta-1})\ \mbox{ as }a\mbox{ goes to }
\infty. $$
Set $\Phi(t):=\psi(t)-C|t|^{-\frac 1\delta}(A_1+iA_2\text{\rm sgn}(t))^{-\frac 1{\delta\beta}}$. 
We have
$$I_2(a) = \int_{\mathbb R} \chi(t)\left\{\hat h(t)\psi(t) 
     - C\, \hat h(0) |t|^{-\frac 1\delta}   (A_1+iA_2\text{\rm sgn}(t))^{-\frac 1{\delta\beta}}\right\}   e^{-i ta}
  \, dt =  J_1(a)+J_2(a)$$
with
$$J_1(a):=\int_{\mathbb R} \chi(t)\left\{\hat h(t)-\hat h(0)\right\}\psi(t)e^{-ita}\, dt 
\quad  \mbox{and}\quad J_2(a):=\hat h(0) \int_{\mathbb R}  \chi(t)\Phi(t)
 e^{-i ta}  \, dt.$$
Note that $J_1(a) = \widehat{g_1}(-a)$, with $g_1 := \chi(\hat h-\hat h(0))\psi$. From Proposition~\ref{series-2} and since $\hat h$ is 
continuously differentiable, we have $\psi(t) = O(|t|^{-1/\delta})$ and $(\hat h(t)-\hat h(0))\psi'(t) = O(|t|^{-1/\delta})$ when $t\rightarrow 0$. Hence $g_1$ and $g_1'$ are Lebesgue-integrable on ${\mathbb R}$, so that we obtain by integration by parts: 
$$J_1(a)=O(a^{-1}) = o(a^{1-1/\delta})\ \mbox{ as }a\mbox{ goes to }
\infty. $$
To study $J_2(a)$, let us set $G(t) := \chi(t)\Phi(t)$, and write 
\begin{eqnarray}
J_2(a) &=& \hat h(0)\int_{\left\{|t| \leq \frac{2\pi}{a}\right\}}
    G(t)\, e^{-ita}\, dt + \hat h(0)
  \int_{\left\{|t|>\frac{2\pi}{a}\right\}} G(t)\, e^{-ita}\, dt  \nonumber \\ 
&=:& \hat h(0)J_{2,1}(a) + \hat h(0)J_{2,2}(a) \label{j2a-sum}
\end{eqnarray}
where $J_{2,1}(a)$ and $J_{2,2}(a)$ are above defined in an obvious way. From Proposition~\ref{series-2} we have $\Phi(t)=\vartheta_0(t)|t|^{-\frac 1\delta}$, with $\lim_{u\rightarrow 0}\vartheta_0(u)=0$. Since $|\chi| \leq 1$, we obtain: 
\begin{equation}\label{int1} 
\big|J_{2,1}(a)\big| \leq \int_{\left\{|t| \leq \frac{2\pi}{a}\right\}}
    \big|\Phi(t)\big|\, dt \leq \frac {2}{1-\frac 1\delta} \left(\frac{2\pi}{a}\right)^{1-\frac 1\delta} 
    \sup_{|t|\leq \frac{2\pi}{a}} |\vartheta_0(t)|  = o(a^{\frac 1\delta -1}),
\end{equation}
as $a$ goes to infinity. 
Next we have $J_{2,2}(a) =  -\int_{\left\{|t| > \frac{2\pi}{a}\right\}} 
  G(t)\,  e^{-i\left(t-\frac{\pi}{a}sgn(t)\right)a}\, dt$, hence 
$$J_{2,2}(a) = \frac 1 2 \left\{\int_{\left\{|t| > \frac{2\pi}{a}\right\}} G(t)\, e^{-ita}\, dt -
     \int_{\left\{|t| > \frac{\pi}{a}\right\}} 
      G\left(t+\frac \pi {a}sgn(t)\right)\, 
   e^{-ita}\, dt\right\},$$
from which we deduce:   
\begin{equation}
\label{J22-one} 
\big|J_{2,2}(a)\big| 
  \leq \frac 1 2 \int_{\left\{|t| > \frac{\pi}{a}\right\}} 
   \left|G(t) - G\left(t+\frac \pi {a}sgn(t)\right)\right|\, dt + 
   \int_{\left\{\frac{\pi}{a} < |t| < \frac{2\pi}{a}\right\}} |G(t)|  \, dt.  
\end{equation}
The last integral in (\ref{J22-one}) is $o(a^{\frac 1\delta-1})$ (use the second inequality in (\ref{int1})). Next, by using Proposition~\ref{series-2}, one can easily see that there exists $\vartheta_1 : {\mathbb R}\setminus\{0\} \rightarrow {\mathbb C}$ such that 
$$G\, '(u) = |u|^{-1-\frac 1\delta}\, \vartheta_1(u) \quad \text{ with }\quad  \lim_{u\rightarrow 0}\vartheta_1(u)=0.$$
Let $\varepsilon>0$, and let $\alpha=\alpha(\varepsilon)>0$ be such that $\sup_{|s|<\alpha}|\vartheta_1(s)|\le 
\frac{\varepsilon}{2\delta}$.  Note that  
$$\left[a>\frac {2\pi} \alpha \text{ and } |t| < \frac{\alpha}{2}\right]\ 
   \Rightarrow\ |t| \leq \left|t+\frac \pi {a}sgn(t)\right| < \alpha.$$
Then, by applying Taylor's inequality to $G$, we obtain for all $a$ such that $a>\frac {2\pi} \alpha$
\begin{equation}
\label{J22-two}
\int_{\left\{\frac{\pi}{a}<|t| < \frac{\alpha}{2}\right\}} 
   \left| G(t) - G\left(t+\frac \pi {a}sgn(t)\right)\right|\, dt 
  \leq  \frac{\varepsilon}{\delta}\, \frac{\pi}{a} \int_{\frac\pi {a}}^{+\infty}
     t^{-1-\frac 1 \delta}\, dt 
  \leq  \varepsilon\left(\frac\pi {a}\right)^{1-\frac 1\delta}. 
\end{equation}
Moreover, since $\Phi$ and $\Phi\, '$ are bounded on ${\mathbb R}\setminus[-\frac{\alpha}{2};\frac{\alpha}{2}]$ (by   Proposition~\ref{series-1-trans}), and from $\chi\in{\mathcal S}({\mathbb R})$, there exists a positive constant $D_\alpha$ such that: 
$$\forall x\in{\mathbb R}\setminus\left[-\frac{\alpha}{2};\frac{\alpha}{2}\right],\quad 
  \big|G\, '(x)\big|
   \leq \frac{D_\alpha}{x^{2}}.$$
Thus, if $a$ is large enough, namely if $a$ is such that $\frac{4D_\alpha}{\alpha}\, (\frac{\pi}{a})^{\frac 1\delta} \leq \varepsilon$, then we have 
\begin{equation}
\label{J22-three}
\int_{\left\{|t| \geq \frac{\alpha}{2}\right\}} \left| G(t) - G\left(t+\frac \pi {a}sgn(t)\right)\right|\, dt 
  \, \leq\,   2 D_\alpha\, \frac{\pi}{a} \int_{\frac{\alpha}{2}}^{+\infty}
     t^{-2}\, dt  
\, \leq\,  \varepsilon\, \left(\frac\pi {a}\right)^{1-\frac 1\delta}
\end{equation}
From (\ref{J22-one}) (\ref{J22-two}) (\ref{J22-three}), it follows that we have when $a$ is sufficiently large: 
$J_{2,2}(a) \leq \varepsilon\, \left(\frac\pi {a}\right)^{1-\frac 1\delta}$. From this fact and from (\ref{j2a-sum}) (\ref{int1}), we have: 
$$J_2(a) = o(a^{1/\delta-1})\ \mbox{ as }a\mbox{ goes to }
\infty.$$
The desired property for $I_2(a)$ is then established. This completes the proof of Theorem~~\ref{threc0}. 
\fdem
\noindent{\bf Remark:}
The generalization of our proof to the more general context when the distribution
of $\xi_0$ belongs to the normal domain of attraction of a stable distribution
of index $\beta$ is not as simple as in the recurrent case.
Indeed we used precise estimation of the derivative of $\psi$ that should require
the existence of the derivative of $\varphi_\xi$ outside $0$, which does not appear
as a natural hypothesis when $\beta<1$ since $\xi_0$ is not integrable.\\*
\\*
\noindent
{\bf Acknowledgments:}\\*
The authors are deeply grateful to Lo\"{\i}c Herv\'{e} for helpful and stimulating discussions.


\begin{thebibliography}{00}


\bibitem{Bla1} Blackwell, {\em A renewal theorem.} Duke Math J. {\bf 15} (1948), 145--150.

\bibitem{Bla2} Blackwell, {\em Extension of a renewal theorem.} Pacific J. Math. {\bf 3} (1953), 315--320.

\bibitem{bolthausen} 
Bolthausen, E. 
{\em A central limit theorem for two-dimensional random walks 
in random sceneries.}
Ann. Probab. {\bf 17} (1989), no. 1, 108--115.


\bibitem{Borodin}
Borodin, A. N. 
{\em A limit theorem for sums of independent random variables defined 
on a recurrent random walk.}
 (Russian)  Dokl. Akad. Nauk SSSR  {\bf 246}  (1979), no. 4, 786--787. 

\bibitem{Borodin1} Borodin, A. N. {\em Limit theorems for sums of independent 
random variables defined on a transient random walk.} Investigations in the theory 
of probability distributions, IV. Zap. Nauchn. Sem. Leningrad. Otdel. Mat. Inst. Steklov. 
(LOMI) {\bf 85} (1979), 17--29, 237, 244.



\bibitem{Bre} Breiman, L. {\em Probability} Classic in Applied Mathematics, SIAM, 1993.



\bibitem{BFFN} Castell, F.; Guillotin-Plantard, N.; P\`ene, F.; Schapira, Br. 
{\em A local limit theorem for random walks in random scenery and on randomly oriented lattices.} 
Ann. Probab. {\bf 39} (2011), no. 6, 2079--2118.

\bibitem{FFN} Castell, F.; Guillotin-Plantard, N.; P\`ene, F. 
{\em Limit theorems for one and two-dimensional random walks in random scenery.} 
To appear in  Ann. Inst. H. Poincar\'e Probab. Statist. (2011).





  
  



\bibitem{DU} Deligiannidis, G.; Utev, S., {\em An asymptotic variance of the self-intersections of random walks}, Sib. Math. J. (2011), Vol 52, No 4, 639--650.

\bibitem{HS} Den Hollander, F.; Steif, J. E., {\em Random walk in random scenery: a survey of some recent results}, Dynamics \& stochastics, 53--65, 
IMS Lecture Notes Monogr. Ser., 48, Inst. Math. Statist., Beachwood, OH, (2006).





\bibitem{EFP} Erd\" {o}s, P.; Feller, W.; Pollard H. {\em A property of power series with positive coefficients.} Bull. Amer. Math. Soc. 55, (1949), 201--204.

\bibitem{Fel} Feller, W. {\em An introduction to probability theory and its applications. Vol. II.} Second edition John Wiley and Sons, Inc., New York-London-Sydney, (1971), xxiv+669 pp.


\bibitem{Gui} Guibourg, D. {\em Th\'eor\`emes de renouvellement pour des fonctionnelles additives associ\'ees \`a des cha\^{\i}nes de Markov fortement ergodiques.} Phd Thesis (2011).








\bibitem{KestenSpitzer} Kesten, H.; Spitzer, F. 
{\em  A limit theorem related to a new class of self-similar processes.} 
Z. Wahrsch. Verw. Gebiete {\bf 50} (1979), no. 1, 5--25.  





\bibitem{ledoussal} Le Doussal, P. {\em  Diffusion in layered random flows, polymers, electrons 
in random potentials, and spin depolarization in random fields.}
  J. Statist. Phys.  {\bf 69}  (1992),  no. 5-6, 917--954.

\bibitem{LGR} Le Gall, J.F.; Rosen, J. {\em The range of stable random walks.} Ann. Probab. {\bf 19} (1991), 650--705.



\bibitem{matheron_demarsily} Matheron, G.; de Marsily G.
{\em Is transport in porous media always diffusive? A counterxample.}
Water Resources Res. {\bf 16} (1980), 901--907. 



\bibitem{KS} Schmidt, K. {\em On recurrence.} Z. Wahrsch. Verw. Gebiete
{\bf 68} (1984), 75-95.

\bibitem{spitzer} Spitzer, F. {\em Principles of random walks.}
Van Nostrand, Princeton, N.J. (1964).

\end{thebibliography}
\end{document}